\documentclass[11pt,reqno]{amsart}
\usepackage{amssymb}
\usepackage[all]{xy}
\newtheorem{thm}{Theorem}[section]
\newtheorem{lem}[thm]{Lemma}
\newtheorem{prop}[thm]{Proposition}

\theoremstyle{definition}
\newtheorem{dfn}[thm]{Definition}
\newtheorem{ex}[thm]{Example}

\theoremstyle{remark}
\newtheorem{remark}[thm]{Remark}

\newtheorem{notation}[thm]{Notation}

\newcommand{\CA}{{\mathcal{A}}}

\newcommand{\CE}{{\mathcal{E}}}

\newcommand{\CF}{{\mathcal{F}}}

\newcommand{\CEa}{\CE^{0,-}}

\newcommand{\bEz}{\overline{\CE^0}}
\newcommand{\bFz}{\overline{\CF^0}}

\newcommand{\CL}{{\mathcal{L}}}
\newcommand{\CB}{{\mathcal{B}}}

\newcommand{\af}{\alpha}
\newcommand{\bt}{\beta}
\newcommand{\gm}{\gamma}
\newcommand{\dt}{\delta}

\newcommand{\ld}{\lambda}

\newcommand{\sk}{{\rm sink}}

\newcommand{\Om}{\Omega}

\begin{document}


\title[AF labeled graph $C^*$-algebras]
{AF labeled graph $C^*$-algebras}

\author[J. A. Jeong]{Ja A Jeong$^{\dagger}$}
\thanks{Research partially supported by NRF-2012-008160$^{\dagger}$
and PARC  2012-000939$^{\ddagger}$.}
\address{
Department of Mathematical Sciences and Research Institute of Mathematics\\
Seoul National University\\
Seoul, 151--747\\
Korea} \email{jajeong\-@\-snu.\-ac.\-kr }

\author[E. J. Kang]{Eun Ji Kang$^{\dagger}$}
\address{
Department of Mathematical Sciences\\
Seoul National University\\
Seoul, 151--747\\
Korea} \email{kkang33\-@\-snu.\-ac.\-kr }

\author[S. H. Kim]{Sun Ho Kim$^{\ddagger}$}
\address{
PARC\\
Seoul National University\\
Seoul, 151--747\\
Korea} \email{hoya4200\-@\-snu.\-ac.\-kr }




\begin{abstract}
It is known that a graph $C^*$-algebra $C^*(E)$ is 
approximately finite dimensional (AF) if and only if 
the graph $E$ has no loops. 
In this paper we consider the question of when a labeled graph $C^*$-algebra 
$C^*(E,\CL,\CB)$ is AF. 
A notion of loop in a labeled space $(E,\CL,\CB)$ 
is defined when $\CB$ is the smallest one among the accommodating sets
that are closed under relative complements  and 
it is proved that if a labeled graph 
$C^*$-algebra is AF,  the labeled space has no loops. 
A sufficient condition for a labeled space to be associated to AF algebra  
is also given. 
For graph $C^*$-algebras $C^*(E)$, this sufficient condition is 
also a necessary one.
Besides, we discuss other equivalent conditions for a graph $C^*$-algebra 
to be AF in the setting of labeled graphs and prove that these conditions are not 
always equivalent by invoking various examples.
\end{abstract}

\maketitle

\setcounter{equation}{0}

\section{Introduction}

The class of graph $C^*$-algebras was introduced 
in \cite{KPR, KPRR} as a generalization of the 
Cuntz-Krieger algebras of finite $\{0,1\}$ matrices \cite{CK}. 
The main benefit of working with  graph algebras lies in the fact that 
many complex properties and structures of   graph $C^*$-algebras 
can be explained  in terms of conditions of graphs   
(see \cite{BHRS, BPRS, JP, KPR, KPRR} among many others).  
 For example, it is now well known \cite{KPR} that 
a directed graph $E$ has no loops if and only if its 
graph $C^*$-algebra $C^*(E)$ is approximately finite dimensional (AF). 
Moreover the class contains 
all AF algebras up to  Morita equivalence \cite{Dr}. 

Besides the graph $C^*$-algebras, there have been various generalizations of 
Cuntz-Krieger algebras.  
The ultragraph algebras \cite{To1} and the Exel-Laca algebras \cite{EL} 
are those generalizations  which also include the $C^*$-algebras 
of row-finite graphs with no sinks. 
In \cite{KST}, conditions for an AF algebra to be realized as a graph $C^*$-algebra, 
an Exel-Laca algebra, and an ultragraph algebra are given,  
and then in \cite{ES}  it is proved that 
if a higher-rank graph algebra $C^*(\Lambda)$ is AF, 
the higher-rank graph $\Lambda$ does not have  
an appropriate analogue of loop. 
The higher-rank graph algebras are of course another generalization of 
the Cuntz-Krieger algebras.

Recently  a class of $C^*$-algebras $C^*(E, \CL, \CB)$ 
associated with labeled graphs $(E,\CL)$, more 
explicitly labeled spaces $(E,\CL,\CB)$, 
has been introduced in \cite{BP1} to provide a common framework for 
working with some of the generalized  Cuntz-Krieger algebras, 
and studied in \cite{BCP, BP2,  JK, JKP}. 
We investigate in this paper the question of when a labeled graph 
$C^*$-algebra $C^*(E,\CL,\bEz)$ is AF, where 
$(E,\CL,\bEz)$ is a labeled space such that the accommodating set 
$\bEz$, consisting of certain vertex subsets, is the smallest one 
that is closed under relative complements. 
To explore the question, 
we first find several conditions on a directed graph $E$  
that are equivalent to the existence of a loop in $E$ 
(Proposition~\ref{prop-AF graph algebra}), and then 
extending one of these conditions we define 
a notion of loop for a labeled space (Definition~\ref{dfn-loop}). 
Each of the other equivalent conditions can also be 
 restated in terms of labeled spaces, but it is not clear 
 whether all these conditions, (a)-(d) stated below, 
 are still equivalent to each other, 
 especially to  $C^*(E,\CL,\bEz)$ being AF:

\vskip .5pc 

\begin{enumerate}
\item[(a)] For every finite set  $\{A_1, \dots, A_N\}$ of $\bEz$ and every $K\geq 1$, 
there exists an $m_0\geq 1$ such that 
$A_{i_1} E^{\leq K} A_{i_2} \cdots E^{\leq K} A_{i_n} =\emptyset$  
for all $n> m_0$ and $A_{i_j}\in \{A_1,\dots,A_N\}$.  
\item[(b)] $(E,\CL,\bEz)$ has no repeatable paths. 
\item[(c)] $C^*(E,\CL,\bEz)$ is an AF algebra.
\item[(d)] $(E,\CL,\bEz)$ has no loops (in the sense of Definition~\ref{dfn-loop}).
\end{enumerate}

\vskip 1pc 
\noindent In (a), $A_{i_1} E^{\leq K} A_{i_2} \cdots E^{\leq K} A_{i_n}$ 
denotes the set of all paths $x=x_1\cdots x_{n-1}$ consisting of 
subpaths $x_k$, from $A_{i_k}$ to $A_{i_{k+1}}$ in $E$, 
with length $|x_k|\leq K$  for $1\leq k\leq n-1$.
A path $\af$ is {\it repeatable} if $\af^n$ appears in 
the (labeled) graph for all $n\geq 1$.

The conditions (a)-(d) above are all equivalent for 
graph $C^*$-algebras $C^*(E)\cong C^*(E,\CL_{id},\bEz)$ 
(Proposition~\ref{prop-AF graph algebra}) as already mentioned, 
where $\CL_{id}$ is the trivial labeling.  
The purpose of this paper is to understand the relations 
of these conditions. 
Both of the implications (a) $\Rightarrow$  (b) and (b) $\Rightarrow$ (d) 
are immediate, but it will turn out in this paper that 
not all of them are equivalent. 
This shows an interesting  contrast between the labeled graph $C^*$-algebras and 
the usual graph $C^*$-algebras.  

The main results obtained in the paper are as follows.

\vskip 1pc 

\begin{thm} {\rm(Theorem~\ref{thm-loopAF})}
 Let $(E,\CL)$ be a labeled graph. 
If $C^*(E,\CL, \bEz)$ is an AF algebra, the labeled space $(E,\CL,\bEz)$ has no loops.
\end{thm}
 
\begin{thm} {\rm (Theorem~\ref{thm-AF})}
Let $(E,\CL,\bEz)$ be a labeled space such that  
for every finite subset  $\{A_1, \dots, A_N\}$ of $\bEz$ and every $K\geq 1$, 
there exists an $m_0\geq 1$ for which 
$$ A_{i_1} E^{\leq K} A_{i_2} E^{\leq K} A_{i_3}\cdots E^{\leq K} A_{i_n} =\emptyset$$ 
for all $n> m_0$ and $1\leq i_j\leq N$. 
Then $C^*(E,\CL,\bEz)$ is an AF algebra. 
\end{thm}

\begin{thm} {\rm (Theorem~\ref{thm-repeatablepath})}
Let $C^*(E,\CL,\bEz)=C^*(s_a,p_A)$ be the 
$C^*$-algebra of a labeled graph $(E,\CL)$ with no sinks or sources. 
Let $C^*(E,\CL,\bEz)$ have a repeatable path $\af\in \CL^*(E)$.   
If $p_{r(\af^m)}$ does not belong to the ideal generated by a projection 
$p_{r(\af^m)\setminus r(\af^{m+1})}$ for some $m\geq 1$, 
$C^*(E,\CL,\bEz)$ is not AF. 
\end{thm}
 
\vskip 1pc 
\noindent
Theorem~\ref{thm-loopAF} and Theorem~\ref{thm-AF} show that 
  (c) $\Rightarrow$ (d) and (a) $\Rightarrow$ (c)  
hold true, respectively, and 
Theorem~\ref{thm-repeatablepath} can be regarded as a partial result 
for (c) $\Rightarrow$ (b). 
The converse of Theorem~\ref{thm-loopAF} 
may not be true (Example~\ref{ex-AF-noloop}(iii)); 
(d) $\nRightarrow$ (c), in general. 
For other implications, we show in Example~\ref{ex-Morse} that 
there is a labeled space $(E,\CL,\bEz)$  with no repeatable paths 
whose $C^*$-algebra $C^*(E,\CL,\bEz)$ is not AF. 
For this, we use the fact that the Morse sequence 
does not contain a block  of the form $\af\af\af'$ for 
a path $\af$ and its initial path $\af'$; thus (b) $\nRightarrow$ (c), in general. 
Furthermore the labeled space of Example~\ref{ex-Morse} does not satisfy 
the condition (a); thus (b) $\nRightarrow$ (a).

To prove Theorem~\ref{thm-loopAF}, 
we need a notion of exit of a loop in a labeled space.
There are, unlike in graph case, three possible types of exits 
 of a loop (Definition~\ref{dfn-loop}).    
We show  in Proposition~\ref{prop-loop-exit} that if  $(E,\CL,\bEz)$ has a loop 
$\af$ with an exit of any type, its associated $C^*$-algebra $C^*(E,\CL,\bEz)$ 
has an infinite projection, 
which extends the same result for graph $C^*$-algebras \cite{KPR}.
After the proofs of our main results are completed, 
we finally discuss relations between disagreeable labeled spaces and
 labeled spaces satisfying some of the conditions (a)-(d) to make 
 the conditions to be understood more clearly. 
 
 The paper is organized as follows. 
In Section 2, we give necessary background on graph and labeled graph 
$C^*$-algebras. Then we begin Section 3  with  characterizations of 
existence of a loop in a directed graph so that they can be directly extended  
to labeled spaces. From this, we define a notion of loop for labeled space. 
The main results on  AF labeled graph $C^*$-algebras are obtained in Section 4.  

\vskip 1pc

\section{Preliminaries}
 
\subsection{\bf Directed graphs and labeled spaces}
We use notational conventions of \cite{KPR} for graphs and
graph $C^*$-algebras and of \cite{BP2} for
labeled spaces and their $C^*$-algebras.
By a  {\it directed graph} we mean a quadruple $E=(E^0,E^1,r_E,s_E)$
consisting of a countable set of
vertices  $E^0$, a countable set of edges $E^1$,
and the range, source maps $r_E$, $s_E: E^1\to E^0$
 (we often write $r$ and $s$ for $r_E$ and $s_E$, respectively).
If a vertex $v\in E^0$ emits (receives, respectively) no edges, 
$v$ is called a {\it sink} ({\it source}, respectively).  
$E^0_\sk$  denotes the set of all sinks of $E$ and 
$E^n$ denotes the set of all finite paths $\ld=\ld_1\cdots \ld_n$
 of {\it length} $n$ ($|\ld|=n$),
($\ld_{i}\in E^1,\ r(\ld_{i})=s(\ld_{i+1}), 1\leq i\leq n-1$).
We write $E^{\leq n}$ and $E^{\geq n}$  
for the sets $\cup_{i=1}^n E^i$ and $\cup_{i=n}^\infty E^i$, respectively.
The maps $r$ and $s$ naturally extend to $E^{\geq 0}$, where 
$r(v)=s(v)=v$ for $v\in E^0$. 

 For a vertex subset $A\subset E^0$,
 $A_\sk$ denotes  the  sinks $A\cap E^0_\sk$ in $A$, and 
for $\mathcal B\subset 2^{E_0}$ we simply denote the set
$\{A_\sk:\, A\in \mathcal B\}$ by $\mathcal B_\sk$.
Also with abuse of notation, for $\mathcal B\subset 2^{E_0}$ and $A\subset E_0$, 
we write 
$$\CB\cap A:=\{ B\in \CB:\, B\subset A\}.$$  
 
 A {\it labeled graph} $(E,\CL)$ over a countable alphabet  $\CA$
 consists of a directed graph $E$ and
 a {\it labeling map} $\CL:E^1\to \CA$.
 We assume that $\CL(E^1)=\CA$. 
 Let $\CA^*$ and $\CA^\infty$ be the sets of all finite sequences
 (of length greater than or equal to $1$)
 and infinite sequences with terms in $\CA$, respectively.
 Then $\CL(\ld):=\CL(\ld_1)\cdots \CL(\ld_n)\in \CA^*$
 for $\ld=\ld_1\cdots \ld_n\in E^n$, and
 $\CL(\dt):=\CL(\dt_1)\CL(\dt_2)\cdots\in \CL(E^\infty)\subset\CA^\infty$
 for $\dt=\dt_1\dt_2\cdots \in E^\infty$.
We use notation $\CL^*(E):=\CL(E^{\geq 1})$.
For $\af=\af_1\af_2\cdots\af_{|\af|}\in \CL^*(E)$,
we denote the subsegment $\af_i\cdots \af_j$ of $\af$
by $\af_{[i,j]}$ for $1\leq i\leq j\leq |\af|$.
A subsegment of the form $\af_{[1,j]}$ is called
an {\it initial path} of $\af$.
 The {\it range} $r(\af)$ and {\it source} $s(\af)$
 of a labeled path $\af\in \CL^*(E)$ are
 subsets of $E^0$ defined by
\begin{align*}
r(\af) &=\{r(\ld) \,:\, \ld\in E^{\geq 1},\,\CL(\ld)=\af\},\\
 s(\af) &=\{s(\ld) \,:\, \ld\in E^{\geq 1},\, \CL(\ld)=\af\}.
\end{align*}
 The {\it relative range of $\af\in \CL^*(E)$
 with respect to $A\subset 2^{E^0}$} is defined to be
$$
 r(A,\af)=\{r(\ld)\,:\, \ld\in E^{\geq 1},\ \CL(\ld)=\af,\ s(\ld)\in A\}.
$$
If $\CB\subset 2^{E^0}$ is a collection of subsets of $E^0$ such that
 $r(A,\af)\in \CB$ whenever $A\in \CB$ and $\af\in \CL^*(E)$,
 $\CB$ is said to be
 {\it closed under relative ranges} for $(E,\CL)$.
 We call $\CB$ an {\it accommodating set} for $(E,\CL)$
 if it is closed under relative ranges,
 finite intersections and unions and contains $r(\af)$ for all $\af\in \CL^*(E)$.
 A set $A\in \CB$ is called {\it minimal} (in $\CB$)  
 if $A$ does not have any proper subset in $\CB$. 

  If $\CB$ is accommodating for $(E,\CL)$, the triple $(E,\CL,\CB)$ is called
 a {\it labeled space}. 
 For $A,B\in 2^{E^0}$ and $n\geq 1$, let
 $$ AE^n =\{\ld\in E^n\,:\, s(\ld)\in A\},\ \
  E^nB=\{\ld\in E^n\,:\, r(\ld)\in B\},$$
 and  $AE^nB=AE^n\cap E^nB$.
 We write $E^n v$ for $E^n\{v\}$ and $vE^n$ for $\{v\}E^n$,
 and will use notation like $AE^{\geq k}$ and $vE^\infty$
 which should have their obvious meaning.
For convenience we also take  conventions like  
$AE^0=A$ and $\CL(A)=A$ for $A\in \CB$, etc.

\vskip 1pc

\begin{notation}\label{notation-paths} 
For $A_i\in \CB$, $1\leq i\leq n$, and $K\geq 1$, we will use the following notation  
$$A_1 E^{\leq K}A_2\cdots E^{\leq K} A_{n+1}$$ 
for the set
$\{ x_1 \cdots x_{n}\in E^{\geq 1}:\, x_i\in A_iE^{\leq K}A_{i+1},\ 1\leq i\leq n\,\}$ 
 of paths  in $E^*$. 
 To stress the fact that a path $x=x_1\cdots x_n$ belongs to 
 $A_1 E^{\leq K}A_2\cdots E^{\leq K} A_{n+1}$, 
 we may write $A_1 x_1 A_2\cdots x_n A_{n+1}$ for $x$.
\end{notation} 
 
\vskip 1pc

 A labeled space $(E,\CL,\CB)$ is said to be {\it set-finite}
 ({\it receiver set-finite}, respectively) if for every $A\in \CB$ and $l\geq 1$ 
 the set  $\CL(AE^l)$ ($\CL(E^lA)$, respectively) is finite.
  A labeled space $(E,\CL,\CB)$ is {\it weakly left-resolving} if
 $$r(A,\af)\cap r(B,\af)=r(A\cap B,\af)$$
 holds for all $A,B\in \CB$ and  $\af\in \CL^*(E)$.
 A labeled space $(E,\CL)$ is  {\it left-resolving} if the
map $ \CL : r^{-1}(v) \rightarrow \mathbf{\CA}$ is injective for each $v \in E^0$, and
{\it label-finite} if $|\CL^{-1}(a)| < \infty $ for each $a \in \CA$.
 If $(E,\CL)$ is left-resolving, then
it is  label-finite if and only if $r(a)$ is 
finite for all $a\in \CA$.

By $\Om_0(E)$ we denote the set of all vertices of $E$ that are not sources.
For $v,w\in \Om_0(E)\subset E^0$,
we write $v\sim_l w$ if $\CL(E^{\leq l} v)=\CL(E^{\leq l} w)$
as in \cite{BP2}.
Then  $\sim_l $ is an equivalence relation on the set $\Om_0(E)$.
The equivalence class $[v]_l$ of $v$  is called a {\it generalized vertex}.
Let $\Om_l(E):=\Om_0(E)/\hskip -.4pc\sim_l$ for $l\geq 1$.
   If $k>l$,  $[v]_k\subset [v]_l$ is obvious and
   $[v]_l=\cup_{i=1}^m [v_i]_{l+1}$
   for some vertices  $v_1, \dots, v_m\in [v]_l$ (\cite[Proposition 2.4]{BP2}). 

\vskip 1pc

\noindent
{\bf Assumptions.}\label{assumptions-sink}
 We assume that a labeled space $(E,\CL,\CB)$
 considered in this paper  always satisfies the following:
\begin{enumerate}
\item[(i)] $(E,\CL,\CB)$ is weakly left-resolving.
\item[(ii)] $(E,\CL,\CB)$  is set-finite and receiver set-finite.
\end{enumerate}

\vskip 1pc
\noindent
Also we assume that if $v\in E^0$ is a  sink, it is not a source. 
A labeled space $(E,\CL,\CB)$ is said to be {\it finite} 
if there are only finitely many generalized vertices $[v]_l$
for each $l\geq 1$, or equivalently 
there are only finitely many labels.

\vskip 1pc

We denote  by $\CE^0$ the smallest accommodating set for $(E,\CL)$ (cf. \cite[p.108]{BP1});
$$\CE^0=\{\cup_{k=1}^m\cap_{i=1}^n  r(\bt_{i,k}):\, \bt_{i,k}\in \CL^*(E)\},
$$ 
and by $\CEa$ the smallest  accommodating set  containing   
$\{r(\alpha): \alpha \in \CL^*(E)\}$ and $\{\{v\}: v\text{ is a sink or a source}\}$.

If $E$ has no sinks or sources, $\CEa =\CE^0$ and  
 every set in $\CE^0$ can be expressed as a finite union of
generalized vertices (\cite[Remark 2.1 and Proposition 2.4.(ii)]{BP2});
$$
\CE^0\subseteq \{\cup_{i=1}^n [v_i]_l:\,  v_i\in \Om_0(E),\  n,l\geq 1\}.
$$
Generalized vertices $[v]_l$ are not always members of the accommodating set 
$\CE^0$ but always the relative complements of sets in $\CE^0$, namely  
$[v]_l=X_l(v)\setminus r(Y_l(v))$,  where
$X_l(v), \,  Y_l(v)$ are given by
$$X_l(v):=\cap_{\af\in \CL(E^{\leq l} v)} r(\af)\ \text{ and }\
Y_l(v):=\cup_{w\in X_l(v)} \CL(E^{\leq l} w)\setminus \CL(E^{\leq l} v)$$
so that  $X_l(v), \,  r(Y_l(v))\in \CE^0$ (\cite[Proposition 2.4]{BP2}).
One can easily check that 
the expression $[v]_l=X_l(v)\setminus r(Y_l(v))$ is valid even for a sink $v$ 
and $[v]_l \cap  r(Y_l(v)) =\emptyset$. 

The accommodating set $\CE^0$ is not necessarily closed under relative complements.
On the other hand, in the construction of  the $C^*$-algebra $C^*(E,\CL,\CB)$
(\cite{BP1, BP2}) of a labeled space $(E,\CL,\CB)$,  
to each nonempty set $A\in \CB$ there is associated  
a nonzero projection $p_A$  in $C^*(E,\CL,\CB)$ in such a manner that
$p_A\leq p_B$ whenever  $A\subset B$.
Hence  $p_B-p_A$ belongs
to $C^*(E,\CL,\CB)$ and it seems reasonable to write $p_{B\setminus A}$ for $p_B-p_A$,  
which leads us to consider accommodating sets that are closed under relative complements.

\vskip 1pc

\begin{notation} Let $(E,\CL)$ be a labeled graph.
\begin{enumerate}

\item[(i)] For a labeled space $(E,\CL,\mathcal B)$, we denote by $\overline{\mathcal B}$ 
the smallest accommodating set that contains 
$\mathcal B\cup  \mathcal B_\sk$ and is closed under relative complements. 
The existence of $\overline{\mathcal B}$ clearly follows 
from considering the intersection of all those 
accommodating sets. 
 $\bEz$ will thus denote the smallest accommodating set 
that is closed under relative complements and contains the sets 
in $\bEz_\sk=\{A_\sk:\, A\in \bEz\}$.

\item[(ii)] As in \cite{BCP}, $\CL^\#(E)$ will denote the union of all labeled paths $\CL^*(E)$ 
and empty word $\epsilon$, 
where  $\epsilon$ is a symbol such that   
$r(\epsilon) =E^0$, $r(A, \epsilon) = A$ for all $A \subset E^0$.

\item[(iii)] If $\CL$ is the identity map $id: E^1\to E^1$,    
it is called the {\it trivial labeling} and will be denote by 
$\CL_{id}$. For a labeled graph $(E,\CL_{id})$, the accommodating set 
$\bEz$ is equal to the collection of all finite subsets of $E^0$. 
\end{enumerate}
\end{notation}

\vskip 1pc

\begin{prop}\label{prop-barE} Let $(E,\CL)$ be a labeled graph and $A\in \bEz$. 
Then $A$ is of the form
$$ A=\big(\cup_{i=1}^{n_1}  [v_i]_l\big) \cup \big(\cup_{j=1}^{n_2}([u_j]_l)_\sk \big)
 \cup \big(\cup_{k=1}^{n_3}  [w_k]_l\setminus ([w_k]_l)_\sk \big)
$$
for some $v_i, u_j, w_k \in \Om_0(E)$ and $l\geq 1$, $n_1, n_2, n_3\geq 0$.
\end{prop}
\begin{proof} Let $\CB$ be the set of all such $A$'s.
Then $\CB\subset \bEz$ is obvious
since $\bEz$ contains all generalized vertices.
Now it suffices to show that $\CB$ is an accommodating set that is
closed under relative complements.
By the proof of \cite[Proposition 2.4]{BP2},
$r(\af)\in\CB$ for all labeled paths $\af\in \CL^*(E)$.
It is easy to see that $\CB$ is closed under finite unions, finite intersections and
relative complements.

In order to show that $\CB$ is closed under
relative ranges, it suffices to see that
$r([v]_l,\af)\in \CB$ for $v\in \Om_0(E)$ and $\af\in \CL^*(E)$.
Since $r([v]_l,\af)\cap r( r(Y_l(v)),\af)=r([v]_l\cap r(Y_l(v)),\af)
=r(\emptyset,\af)=\emptyset$,
we have
$$r([v]_l,\af)=r\big(X_l(v)\setminus r(Y_l(v)), \af\big)
=r(X_l(v),\af)\setminus r\big( r(Y_l(v)),\af\big)$$
which belongs to $\CB$
since $r(X_l(v),\af),\, r( r(Y_l(v)),\af)\in \CE^0 \subset \CB$ and
$\CB$ is closed under relative complements.
\end{proof}

\vskip 1pc

\subsection{Labeled graph $C^*$-algebras}

\begin{dfn} (cf. \cite[Definition 4.1]{BP1} and \cite[Remark 3.2]{BP2})
\label{def-representation}
Let $(E,\CL,\CB)$ be a labeled space such that
$\bEz\subset \CB$.
A {\it representation} of $(E,\CL,\CB)$
consists of projections $\{p_A\,:\, A\in \CB\}$ and
partial isometries
$\{s_a\,:\, a\in \CA\}$ such that for $A, B\in \CB$ and $a, b\in \CA$,
\begin{enumerate}
\item[(i)]  $p_{\emptyset}=0$, $p_Ap_B=p_{A\cap B}$, and
$p_{A\cup B}=p_A+p_B-p_{A\cap B}$,
\item[(ii)] $p_A s_a=s_a p_{r(A,a)}$,
\item[(iii)] $s_a^*s_a=p_{r(a)}$ and $s_a^* s_b=0$ unless $a=b$,
\item[(iv)] for each $A\in \CB$,
$$p_A=\sum_{a\in \CL(AE^1)} s_a p_{r(A,a)}s_a^* + p_{{}_{A_\sk}}.$$
\end{enumerate}
\end{dfn}

\vskip 1pc

\begin{remark}\label{remark-elements}
Let $(E,\CL,\CB)$ be a labeled space such that $\bEz\subset \CB$.
\begin{enumerate}
\item[(i)] The proof of \cite[Theorem 4.5]{BP1} shows that 
there exists a $C^*$-algebra $C^*(E,\CL,\CB)$
generated by a universal representation
$\{s_a,p_A\}$ of $(E,\CL,\CB)$; 
we need to modify the proof slightly, namely
we should mod out the $*$-algebra $k_{(E,\CL,\CB)}$ by the ideal
$J$ generated by the elements
$q_{A\cup B}-q_A -q_B+ q_{A\cap B}$ and
$q_A-\sum_{a\in \CL(AE^1)} s_a q_{r(A,a)}s_a^* - q_{A_\sk}$  for $A,\, B\in \CB$.
If $\{s_a,p_A\}$ is a universal representation of $(E,\CL,\CB)$,
we simply write $C^*(E,\CL,\CB)=C^*(s_a,p_A)$
and call $C^*(E,\CL,\CB)$ the {\it labeled graph $C^*$-algebra} of
a labeled space $(E,\CL,\CB)$.
 Note that   $s_a\neq 0$ and $p_A\neq 0$ for $a\in \CA$
and  $A\in \CB$, $A\neq \emptyset$, and that
$s_\af p_A s_\bt^*\neq 0$ if and only if $A\cap r(\af)\cap r(\bt)\neq \emptyset$.
By Definition~\ref{def-representation}(iv)
and \cite[Lemma 4.4]{BP1} saying that with 
$s_\af:=p_\af$ for $\af\in \CB$,
$$(s_\af p_{A} s_\bt^*)(s_\gm p_{B} s_\dt^*)=\left\{
                      \begin{array}{ll}
                        s_{\af\gm'}p_{r(A,\gm')\cap B} s_\dt^*, & \hbox{if\ } \gm=\bt\gm' \\
                        s_{\af}p_{A\cap r(B,\bt')} s_{\dt\bt'}^*, & \hbox{if\ } \bt=\gm\bt'\\
                        s_\af p_{A\cap B}s_\dt^*, & \hbox{if\ } \bt=\gm\\
                        0, & \hbox{otherwise,}
                      \end{array}
                    \right.
$$ for $\af,\bt,\gm,\dt\in  \CL^{\#}(E)$ and $A,B\in \CB$,
it follows that  
\begin{eqnarray}\label{eqn-elements}
C^*(E,\CL,\CB)=\overline{\rm span}\{s_\af p_A s_\bt^*\,:\,
\af,\,\bt\in  \CL^{\#}(E),\ A\in \CB\}, 
\end{eqnarray}  
where $s_\epsilon$ denotes the unit of 
the multiplier algebra of $C^*(E,\CL,\mathcal B)$ \cite{BCP}. 
It is observed in \cite{JKP} that if $E$ has no sinks nor sources, then 
$$C^*(E,\CL,\CEa)\cong C^*(E,\CL,\bEz).$$
(We wrote $\overline\CE$ for $\bEz$ in the paper \cite{JKP}.)

\item[(ii)] Universal property of  $C^*(E,\CL,\CB)=C^*(s_a, p_A)$ 
defines a strongly continuous action
$\gm:\mathbb T\to Aut(C^*(E,\CL,\CB))$,  called
the {\it gauge action}, such that
$$\gm_z(s_a)=zs_a \ \text{ and } \  \gm_z(p_A)=p_A$$ 
for $a\in \CL(E^1)$ and $A\in \CB$.

\item[(iii)] From Definition~\ref{def-representation}(iv),
 we have for each $n\geq 1$,
$$ p_A=\sum_{\af\in\CL(AE^n)}s_\af  p_{r(A,\af)}s_\af^* +
    \sum_{\gm\in\CL(AE^{\leq n-1})} s_\gm p_{r(A,\gm)_\sk} s_\gm^*,$$
where $\sum_{\gm\in\CL(AE^0)} s_\gm p_{r(A,\gm)_\sk} s_\gm^*:= p_{{}_{A_\sk}}$.
In fact,
\begin{align*}
    p_A   
    =\ &\sum_{a\in \CL(AE^1)} s_a p_{r(A,a)}s_a^* +p_{{}_{A_\sk}} \\
    =\ &\sum_{a\in \CL(AE^1)} s_a \Big(\sum_{b\in \CL(r(A,a)E^1)} s_b p_{r(A,ab)}s_b^*
    + p_{r(A,a)_\sk} \Big)s_a^* + p_{{}_{A_\sk}} \\
    =\ &\sum_{\gm\in\CL(AE^2)}s_\gm  p_{r(A,\gm)}s_\gm^* +
    \sum_{a\in \CL(AE^1)}s_a p_{r(A,a)_\sk} s_a^*  + p_{{}_{A_\sk}}\\
   =\ &\sum_{\gm\in\CL(AE^2)}s_\gm \Big(\sum_c s_c p_{r(A,\gm c)}s_c^*
    +  p_{r(A,\gm)_\sk} \Big)s_\gm^*\\
    &\ \ \ + \sum_{a\in \CL(AE^1)}s_a p_{r(A,a)_\sk} s_a^*  + p_{{}_{A_\sk}}\\
    =\ & \sum_{\af\in\CL(AE^3)} s_\af  p_{r(A,\af)}s_\af^* +
    \sum_{\gm\in\CL(AE^2)} s_\gm p_{r(A,\gm)_\sk} s_\gm^*\\  
    &\ \ \ + \sum_{a\in \CL(AE^1)}s_a p_{r(A,a)_\sk} s_a^*  +  p_{{}_{A_\sk}}\\
    =\ &\ \cdots\\
    =\ & \sum_{\af\in\CL(AE^n)}s_\af  p_{r(A,\af)}s_\af^* +
    \sum_{\gm\in\CL(AE^{\leq n-1})} s_\gm p_{r(A,\gm)_\sk} s_\gm^*.
\end{align*}

\item[(iv)] For $B\in \bEz$, one can easily show that 
the ideal  $I_B$  of $C^*(E,\CL,\bEz)=C^*(s_a,p_A)$ generated by 
the projection $p_B$ is equal to
\begin{eqnarray}\label{eqn-ideal-elements}
I_B=\overline{\rm span}\{s_\af p_A s_\bt^*\,:\,
\af,\,\bt\in  \CL^{\#}(E),\ A\in \bEz\cap r(\CL(BE^{\geq 0}))\,\},
\end{eqnarray}
where $r(\CL(BE^{0})):=B$.
\end{enumerate}
\end{remark}

\vskip 1pc

\section{Loops in labeled spaces}

\subsection{Loops in directed graphs} 
Recall that a path $x \in E^{\geq 1}$ in a directed graph $E$ 
is called a {\it loop} (or a {\it directed cycle}) if 
$s(x)=r(x)$, that is, if it comes back to its source vertex. 
It is well known \cite[Theorem 2.4]{KPR} 
that for a graph $C^*$-algebra $C^*(E)$ to be AF 
it is a sufficient and necessary condition that  $E$ has no loops. 

Since the accommodating set $\bEz$ of a labled graph $(E,\CL_{id})$ 
with the trivial labeling $\CL_{id}$  
contains all the single vertex sets $\{v\}$, $v\in E^0$,  
the following are equivalent  for a path 
$x=x_1\cdots x_m\in E^{\geq 1}(=\CL_{id}^*(E))$: 
\begin{enumerate} 
\item[(i)] $x$ is a loop in $E$,
\item[(ii)] $\{r(x)\}=r(\{r(x)\}, x)$,
\item[(iii)] $x$ is {\it repeatable}, that is, $x^n\in E^{\geq 1}$ for all $n\geq 1$, 
\item[(iv)]  
 $(A_{1}x_1 A_{2}x_2\cdots A_m x_m)^n(A_{1}x_1 A_{2}x_2\cdots A_i x_i) 
 \in \CL_{id}^*(E)$ for all $n\geq 1$ and $1\leq i\leq m$, 
 where $A_{i}= \{ s(x_i)\}\in \bEz$. 
 (See Notation~\ref{notation-paths} for the meaning of 
 $A_{1}x_1 A_{2}x_2\cdots A_m x_m$.)  
\end{enumerate}
From this we can obtain several equivalent conditions 
for a graph $C^*$-algebra $C^*(E)$ to be AF as follows.

\vskip 1pc 

\begin{prop}\label{prop-AF graph algebra} 
Let $(E,\CL_{id},\bEz)$ be a labeled space with the trivial labeling $\CL_{id}$  
so that $C^*(E,\CL_{id},\bEz) \cong C^*(E)$. Then the following are equivalent: 
\begin{enumerate} 
\item[(i)] $C^*(E,\CL_{id},\bEz)$ is AF,
\item[(ii)] $E$ has no loops, 
\item[(iii)] $A\not\subset r(A,x)$ for all $A\in \bEz$ and $x\in \CL_{id}^*(E)$, 
\item[(iv)] there are no repeatable paths in $\CL_{id}^*(E)$, 
\item[(v)] if $\{A_1,\dots , A_m\}$ is a finite collection of sets from 
$\bEz$ and $K\geq 1$, there is an $m_0\geq 1$ such that 
$A_{i_1}E^{\leq K}A_{i_2}\cdots E^{\leq K}A_{i_{n+1}}= \emptyset$ for all $n> m_0$.
\end{enumerate}  
\end{prop} 
\begin{proof} 
We only prove that (ii) and  (v) are equivalent since the equivalence of 
(i) and (ii) is well known and the other implications are rather obvious. 
Suppose $x$ is a loop in $E$, then with $A=\{s(x)\}\in \bEz$ and $K:=|x|\geq 1$, 
it is immediate that (v) does not hold. 
For the converse, suppose that (v) dose not hold and so  there are 
finitely many sets $A_1, \dots, A_m$ in $\bEz$ and $K\geq 1$ such that 
$A_{i_1}E^{\leq K}A_{i_2}\cdots E^{\leq K}A_{i_{n+1}}\neq \emptyset$ 
for all $n\geq 1$. 
Since every set in $\bEz$ is finite, the number of 
vertices in $\cup_{i=1}^m A_i$ is also finite. 
Choose an integer $N$ with $|\cup_{i=1}^m A_i|<N$.   
Then for any path in 
$A_{i_1}E^{\leq K}A_{i_2}\cdots E^{\leq K}A_{i_{N+1}}(\neq \emptyset)$,   
there is a  vertex in $\cup_{i=1}^m A_i$ the path passes through at least two times,   
which proves the existence of a loop (at that vertex) in $E$. 
\end{proof}

\vskip 1pc 
 
\subsection{Loops in labeled spaces}       
If $(E,\CL_{id},\bEz)$ is a labeled space with the trivial labeling, 
by Proposition~\ref{prop-AF graph algebra}, 
the condition that there is a set $A\in \bEz$ satisfying  
$A\subset r(A,x)$ for a path $x$ (in fact, $A=\{s(x)\}=\{r(x)\}$) is equivalent to 
the existence of a loop in $E$.  
This equivalent condition can be extended to a labeled space as follows. 

\vskip 1pc

\begin{dfn}\label{dfn-loop} 
Let  $(E,\CL,\CB)$ be a labeled space and $\af\in \CL^*(E)$ be a labeled path.  
\begin{enumerate} 
\item[(a)] $\af$ is called a {\it generalized loop} at $A\in \CB$ 
if $\af\in \CL(AE^{\geq 1}A)$. 
\item[(b)] $\af$ is called a {\it loop} at $A\in \CB$ if it is a 
generalized loop such that $A\subset r(A,\af)$. 
\item[(c)] A loop  $\af$  at $A\in \CB$ has an {\it exit} if one of the following holds:
\begin{enumerate}
\item[(i)] $\{\af_{[1,k]}:\, 1\leq k\leq |\af| \}\subsetneq \CL(AE^{\leq |\af|})$,
\item[(ii)]  $r(A, \af_{[1,i]})_\sk\neq \emptyset$ for some $i=1, \dots, |\af|$,
\item[(iii)] $A\subsetneq r(A,\af)$.
\end{enumerate}
\end{enumerate}
\end{dfn}

\noindent Note that every loop  $\af$ is {\it repeatable}, that is, 
 $\af^n\in \CL^*(E)$ for all $n\geq 1$ (\cite[Definition 6.6]{BP2}), 
 and every repeatable path is a generalized loop at its range. 
 Not every repeatable path is a loop as we can see in Example~\ref{ex-AF-noloop}(iii). 

\vskip 1pc
 
\begin{remark} Let $(E,\CL,\CB)$ be a labeled space and $A\in \CB$.
\begin{enumerate}
\item[(i)] A generalized loop $\af$ at a minimal set  $A\in \CB$ is 
always a loop because $A\subset r(A,\af)$ follows from 
the minimality of $A$ since $\emptyset\neq A\cap r(A,\af)\subset A$.   
A labeled graph $(E,\CL)$ might have a loop $\af$ 
even when the graph $E$ itself 
has no loops at all 
as we will see in Example~\ref{ex-AF-noloop}(i) and (ii).

\item[(ii)] If $\af$ is a loop at $A$, 
then evidently $p_A\leq p_{r(A,\af)}$.
\end{enumerate}
\end{remark}

\vskip 1pc

\begin{ex} We give examples of labeled graphs with a loop that has an exit. 
 
\begin{enumerate}
\item[(i)] The loop $\af:=b_1b_2$ at $A:=\{v\}\in \bEz$  
has an exit of type (i) of  Definition~\ref{dfn-loop}(c) because
$\{\af_{[1,k]}: 1\leq k\leq 2\}=\{ b_1, b_1b_2 \}$ while 
$\CL(AE^{\leq |\af|})=\{b_1, b_1b_2, b_1 c\}$.
\vskip 1pc

\hskip 5pc 
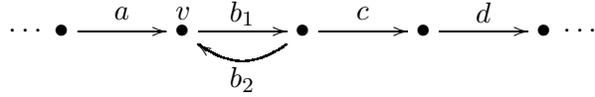
\xy /r0.38pc/:(-13,0)*+{\cdots};
(-10,0)*+{\bullet}="V-1"; (0,0)*+{\bullet}="V0";
(10,0)*+{\bullet}="V1"; (20,0)*+{\bullet}="V2";(30,0)*+{\bullet}="V3";
 "V1";"V0"**\crv{(10,0)&(5,-5)&(0,0)};?>*\dir{>}\POS?(.5)*+!D{};
 "V-1";"V0"**\crv{(-10,0)&(0,0)};?>*\dir{>}\POS?(.5)*+!D{};
 "V0";"V1"**\crv{(0,0)&(10,0)};?>*\dir{>}\POS?(.5)*+!D{};
 "V1";"V2"**\crv{(10,0)&(20,0)};?>*\dir{>}\POS?(.5)*+!D{};
 "V2";"V3"**\crv{(20,0)&(30,0)};?>*\dir{>}\POS?(.5)*+!D{};
(33,0)*+{\cdots};
 (-5,1.5)*+{a};(5,1.5)*+{b_1};
 (15,1.5)*+{c}; (25,1.5)*+{d};
 (0.1,1.5)*+{v};(5,-4)*+{b_2};
\endxy

\vskip 1pc

\item[(ii)] Let $A:=\{v,w\}$. 
Since $A=r(A,b)$, $b$ is a loop at $A$ and  
has an exit of type (ii) of  Definition~\ref{dfn-loop}(c);  
 $r(A, b)_\sk=\{w\}\neq \emptyset$.

\vskip 1pc

\hskip 7pc 
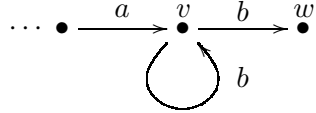
\xy /r0.38pc/:(-13,0)*+{\cdots};
(-10,0)*+{\bullet}="V-1"; (0,0)*+{\bullet}="V0";
(10,0)*+{\bullet}="V1";  
 "V-1";"V0"**\crv{(-10,0)&(0,0)};?>*\dir{>}\POS?(.5)*+!D{};
 "V0";"V1"**\crv{(0,0)&(10,0)};?>*\dir{>}\POS?(.5)*+!D{};
 "V0";"V0"**\crv{(0,0)&(-5,-4)&(0,-8)&(5,-4)&(0,0)};?>*\dir{>}\POS?(.5)*+!D{};
 (-5,1.5)*+{a};(5,1.5)*+{b};
 (0.1,1.5)*+{v};(5,-4)*+{b};
 (10.1,1.5)*+{w}; 
\endxy

\vskip 1pc
\item[(iii)] The loop $\af:=bc$ at $A:=\{v\}\in \bEz$  
has an exit of type (iii) of  Definition~\ref{dfn-loop}(c) 
because  $A\subsetneq r(A,\af)$.

\vskip 1pc

\hskip 5pc 
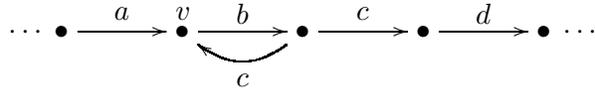
\xy /r0.38pc/:(-13,0)*+{\cdots};
(-10,0)*+{\bullet}="V-1"; (0,0)*+{\bullet}="V0";
(10,0)*+{\bullet}="V1"; (20,0)*+{\bullet}="V2";(30,0)*+{\bullet}="V3";
 "V1";"V0"**\crv{(10,0)&(5,-5)&(0,0)};?>*\dir{>}\POS?(.5)*+!D{};
 "V-1";"V0"**\crv{(-10,0)&(0,0)};?>*\dir{>}\POS?(.5)*+!D{};
 "V0";"V1"**\crv{(0,0)&(10,0)};?>*\dir{>}\POS?(.5)*+!D{};
 "V1";"V2"**\crv{(10,0)&(20,0)};?>*\dir{>}\POS?(.5)*+!D{};
 "V2";"V3"**\crv{(20,0)&(30,0)};?>*\dir{>}\POS?(.5)*+!D{};
(33,0)*+{\cdots};
 (-5,1.5)*+{a};(5,1.5)*+{b};
 (15,1.5)*+{c}; (25,1.5)*+{d};
 (0.1,1.5)*+{v};(5,-4)*+{c};
\endxy

\end{enumerate}
\end{ex}
 
\vskip 1pc

The following proposition is an  extended version  of the fact that 
if a directed graph $E$ has a loop with an exit, its graph $C^*$-algebra 
has an infinite projection. 

\vskip 1pc 

\begin{prop}\label{prop-loop-exit}
Let $(E,\CL)$ be a labeled graph and $\af$ be a loop at $A\in \bEz$ with an exit. 
Then $p_A$  is an infinite projection in $C^*(E,\CL,\bEz)$.
\end{prop}

\begin{proof}
If $A\subsetneq r(A,\af)$, the projection $p_{r(A,\af)}$ is infinite because  
$$p_{r(A,\af)}> p_A\geq s_\af p_{r(A,\af)}s_\af^*\sim p_{r(A,\af)}.$$
If either $\CL(AE^{\leq |\af|}) \supsetneq \{\af_{[1,k]}:\, 1\leq k\leq |\af|  \}$ or 
$r(A, \af_{[1,i]})_\sk\neq \emptyset$ for some $i$, $1\leq i\lneq |\af|$,
by Remark~\ref{remark-elements}(iii) we have
$$ p_A  =\sum_{\bt\in \CL(AE^{|\af|})}s_\bt  p_{r(A,\bt)} s_\bt^*+
\sum_{|\gm|\leq |\af|-1} s_\gm p_{r(A,\gm)_\sk} s_\gm^*  
  \, \gneq s_\af p_{r(A,\af)}s_\af^*.$$  
Thus  $p_A > s_\af p_{r(A,\af)}s_\af^* \sim p_{r(A,\af)}\geq p_A$ 
and we see that the projection $p_{r(A,\af)}$ (hence $p_A$) is infinite.
Now it remains to prove the assertion in case $r(A,\af)_\sk\neq \emptyset$ 
and $A=r(A,\af)$. 
The set $A_0:=A\setminus A_\sk\,(\neq \emptyset)$ then satisfies 
$A_0\subsetneq A=r(A,\af)=r(A_0,\af)$, and by the first argument of the proof 
$p_{A_0}$ is infinite. Hence $p_A(\gneq p_{A_0})$ is infinite.
\end{proof}

\vskip 1pc

\begin{remark}\label{rmk-loops}
A generalized version of Proposition~\ref{prop-loop-exit} 
is also true: Let $(E,\CL)$ be a labeled graph and $\af_1, \dots, \af_n$ be  
distinct labeled paths  with the same length, say $l\geq 1$, 
such that $A\subseteq \cup_{i=1}^n r(A,\af_i)$.   
Then $p_A$  is an infinite projection in $C^*(E,\CL,\bEz)$ if one of the following 
holds:
\begin{enumerate} 
\item[(i)] $\cup_{i=1}^n\{\af_i': \af_i' \text{ is an initial path of } \af_i \}
\subsetneq  \CL(AE^{\leq l})$
\item[(ii)] $r(A,\af_i')_\sk\neq \emptyset$ for some $i$ and an 
initial path $\af_i'$ of $\af_i$
\item[(iii)] $A\subsetneq \cup_{i=1}^n r(A,\af_i)$.
\end{enumerate}

\noindent  
To prove this, first assume the case (iii) and   
set $A_1:=r(A,\af_1)$ and 
$A_i:=r(A,\af_i)\setminus \cup_{j=1}^{i-1}r(A,\af_j)$, $i=2,\dots, n$, so that 
$\cup_{i=1}^n r(A,\af_i)=\cup_{i=1}^n A_i$ is the union of disjoint sets $A_i$'s. 
Then we have 
$$
p_A \geq \sum_{i=1}^n s_{\af_i}p_{r(A,\af_i)}s_{\af_i}^* 
      \geq \sum_{i=1}^n s_{\af_i}p_{A_i}s_{\af_i}^* 
      \sim 
      \sum_{i=1}^n p_{A_i}=p_{\cup A_i}=p_{\cup_{i=1}^n r(A,\af_i)}
      \gneq p_A$$
and so the projection $p_A$ is infinite, 
where the equivalence is given by the partial isometry 
$u:=\sum_{i=1}^n s_{\af_i}p_{A_i}$. 
It is not hard to see that 
the same argument in the proof of Proposition~\ref{prop-loop-exit} 
shows  the assertion for the rest cases. 
\end{remark}

\vskip 1pc
 
\begin{prop}\label{prop-finite}
Let $(E,\CL,\bEz)$ be a labeled space such that 
$C^*(E,\CL,\bEz)$ has no infinite projections. 
Let  $A\in \bEz$ admit a loop.
Then there exists a loop $\af$ at $A$ such that $A=r(A,\af)$ and 
$\CL(AE^{\geq 1})
=\{ \af^k\af': k\geq 0,\ \af' \text{ is an initial path of }\af\}$. 
\end{prop}

\begin{proof} 
Choose a loop $\af$ at $A$  with the smallest length; 
$|\af|\leq |\gm|$ for all loops  $\gm$ at $A$.  
Since $C^*(E,\CL,\bEz)$ has no infinite projections, 
$\af$ does not have an exit by Proposition~\ref{prop-loop-exit}, hence $A=r(A,\af)$ and 
\begin{eqnarray}\label{eqn-loops}
\CL(AE^{\leq |\af|})= \{\af_{[1,k]}:\, 1\leq k\leq |\af| \}. 
\end{eqnarray}  
Now let $\bt\in \CL(AE^{\geq 1})$ be a path with $|\bt|> |\af|$. 
Then by (\ref{eqn-loops}), $\CL(AE^{|\af|})=\{\af\}$ and so 
we can write $\bt =\af\bt'$ for a path $\bt'$.
But then from $A=r(A,\af)$, $\bt'$ must be either an initial path of $\af$ 
or of the form $\af\bt''$ for some path $\bt''$. 
Applying the argument repeatedly, we  finally end up with 
$\bt=\af^k \af'$ for some $k\geq 1$ and an initial path $\af'$ of $\af$. 
\end{proof}

\vskip 1pc

\section{AF labeled graph $C^*$-algebras}

\vskip 1pc 

\begin{remark}\label{rmk-AF-eq} We will consider the following 
properties (a)-(d) of a labeled space $(E,\CL,\bEz)$ 
and its $C^*$-algebra $C^*(E,\CL,\bEz)$.  
These properties are equivalent if $\CL$ is 
the trivial labeling $\CL_{id}$ 
as we have seen in  Proposition~\ref{prop-AF graph algebra}. 
\begin{enumerate}
\item[(a)] For every finite set  $\{A_1, \dots, A_N\}$ of $\bEz$ and every $K\geq 1$, 
there exists an $m_0\geq 1$ such that 
$A_{i_1} E^{\leq K} A_{i_2} \cdots E^{\leq K} A_{i_n} =\emptyset$  
for all $n> m_0$ and $A_{i_j}\in \{A_1,\dots,A_N\}$.  
\item[(b)] $(E,\CL,\bEz)$ has no repeatable paths. 
\item[(c)] $C^*(E,\CL,\bEz)$ is an AF algebra.
\item[(d)] $(E,\CL,\bEz)$ has no loops (in the sense of Definition~\ref{dfn-loop}).
\end{enumerate} 

\noindent 
Note that (a) $\Rightarrow$ (b) 
follows from a simple  observation that 
if $\af$ is a repeatable path,
then with $A:=r(\af)$ one has
$A_{i_1}E^{|\af|}A_{i_2}\cdots E^{|\af|}A_{i_n}\neq \emptyset$ for all $n\geq 1$, 
where $A_{i_j}=A$, $j=1,\dots, n$.  
The implication (b) $\Rightarrow$ (d) is obvious.  
 
For the other implications, we shall see 
(b) $\nRightarrow$ (a) and (b) $\nRightarrow$ (c), in general   
throughout Example~\ref{ex-Morse}. 
Consequently  (d) $\nRightarrow$ (c) follows although 
it can  also be seen from Example~\ref{ex-AF-noloop}(iii).
We will show that (c) $\Rightarrow$ (d) and  (a) $\Rightarrow$ (c) 
hold true in Theorem~\ref{thm-loopAF} and Theorem~\ref{thm-AF}, respectively.    
  
It would be interesting to know whether 
the remaining implication (c) $\Rightarrow$ (b) is true, that is, 
whether $C^*(E,\CL,\bEz)$ will never be AF 
whenever $(E,\CL,\bEz)$ contains a repeatable path. 
In Theorem~\ref{thm-repeatablepath}, we obtain 
a partial affirmative answer.
\end{remark}

\vskip 1pc
  
\begin{thm}\label{thm-loopAF} Let $(E,\CL)$ be a labeled graph. 
If $C^*(E,\CL, \bEz)$ is an AF algebra, the labeled space $(E,\CL,\bEz)$ has no loops.
\end{thm}

\begin{proof} 
Suppose, for contradiction, that $(E,\CL,\bEz)$ has a loop $\af$ at $A\in \bEz$. 
 By Proposition~\ref{prop-loop-exit}, $A= r(A,\af)$  and so  
$p_A s_\af=s_\af p_{r(A,\af)}=s_\af p_A$. 
Then $U:=s_\af p_A$ satisfies  
$$p_A=U^*U\sim UU^*=s_\af p_A s_\af^*=s_\af p_{r(A,\af)} s_\af^*\leq p_A.$$ 
Since $p_A$ is a finite projection, it follows that  
$U$ is a unitary of the unital hereditary subalgebra $p_A C^*(E,\CL,\bEz)p_A$. 
Since $\gm_z(p_A)=p_A$ for any $z\in \mathbb T$, 
the algebra $p_A C^*(E,\CL,\bEz)p_A$ admits an action of $\mathbb T$ which is 
the restriction of the gauge action  $\gm$ on $C^*(E,\CL,\bEz)$.
Then the fact that $\gm_z(U)=\gm_z(s_\af)p_A=z^{|\af|}U$ shows that 
$U$ is not in the unitary path connected
component of the unit $p_A$ (\cite[Proposition 3.9]{ES}), which is a contradiction to 
the assumption that $C^*(E,\CL,\bEz)$ (hence any nonzero hereditary subalgebra) is 
an AF algebra.
\end{proof}

\vskip 1pc

In Example~\ref{ex-AF-noloop}(iii) below, 
we see that the converse of Theorem~\ref{thm-loopAF} 
may not be true, in general.

\vskip 1pc

\begin{ex}\label{ex-AF-noloop} 
(i) For the following labeled graph $(E,\CL)$
\vskip 1pc

\hskip 5pc 
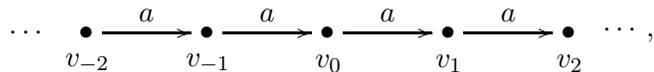
\xy /r0.38pc/:(-25,0)*+{\cdots};(25,0)*+{\cdots ,};
 (-20,0)*+{\bullet}="V-2";
(-10,0)*+{\bullet}="V-1"; (0,0)*+{\bullet}="V0";
(10,0)*+{\bullet}="V1"; (20,0)*+{\bullet}="V2";
 "V-2";"V-1"**\crv{(-20,0)&(-10,0)};
 ?>*\dir{>}\POS?(.5)*+!D{};
 "V-1";"V0"**\crv{(-10,0)&(0,0)};
 ?>*\dir{>}\POS?(.5)*+!D{};
 "V0";"V1"**\crv{(0,0)&(10,0)};
 ?>*\dir{>}\POS?(.5)*+!D{};
 "V1";"V2"**\crv{(10,0)&(20,0)};
 ?>*\dir{>}\POS?(.5)*+!D{};
 (-15,1.5)*+{a};(-5,1.5)*+{a};(5,1.5)*+{a};
 (15,1.5)*+{a};
 (0.1,-2.5)*+{v_0};(10.1,-2.5)*+{v_1};
 (-9.9,-2.5)*+{v_{-1}};
 (-19.9,-2.5)*+{v_{-2}}; 
 (20.1,-2.5)*+{v_{2}}; 
\endxy

\vskip 1pc

\noindent
we have $\bEz=\{r(a)\}=\{E^0\}$ and the path $a$ is a loop at $r(a)$.
By Theorem~\ref{thm-loopAF},  
$C^*(E,\CL,\bEz):=C^*(s_a, p_A)$ is not AF. 
Actually
$C^*(E,\CL,\bEz)\cong C(\mathbb{T})$ is the universal $C^*$-algebra 
generated by the unitary $s_a$. 

\vskip 1pc 
\noindent  
(ii) $\bEz$ of the  following labeled graph consists of 
three sets $r(a)=E^0$, $r(a)_\sk=\{v_0\}$, and 
 $A:=r(a)\setminus r(a)_\sk=\{v_{-1}, v_{-2}, \dots\}$. 
\vskip 1pc

\hskip 5pc 
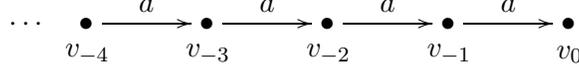
\xy /r0.38pc/:(-15,0)*+{\cdots};
(-10,0)*+{\bullet}="V-1"; (0,0)*+{\bullet}="V0";
(10,0)*+{\bullet}="V1"; (20,0)*+{\bullet}="V2";
(30,0)*+{\bullet}="V3";
 "V-1";"V0"**\crv{(-10,0)&(0,0)};
 ?>*\dir{>}\POS?(.5)*+!D{};
 "V0";"V1"**\crv{(0,0)&(10,0)};
 ?>*\dir{>}\POS?(.5)*+!D{};
 "V1";"V2"**\crv{(10,0)&(20,0)};
 ?>*\dir{>}\POS?(.5)*+!D{};
 "V2";"V3"**\crv{(20,0)&(30,0)};
 ?>*\dir{>}\POS?(.5)*+!D{};
 (-5,1.5)*+{a};(5,1.5)*+{a};
 (15,1.5)*+{a};(25,1.5)*+{a};
 (0.1,-2.5)*+{v_{-3}};(10.1,-2.5)*+{v_{-2}};
 (-9.9,-2.5)*+{v_{-4}};
 (20.1,-2.5)*+{v_{-1}};(30.1,-2.5)*+{v_{0}};
\endxy

\vskip 1pc
\noindent  
 Since $A\subsetneq r(A,a)$, 
 the loop $a$ at $A$ has an exit  and 
 $C^*(E,\CL,\bEz)$ contains an infinite projection 
 by Proposition~\ref{prop-loop-exit}. 

\vskip 1pc
\noindent  
(iii) If $(E,\CL)$ is as follows 
\vskip 1pc

\hskip 7pc 
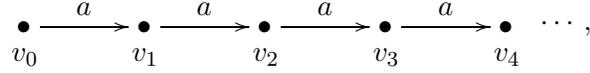
\xy /r0.38pc/:(15,0)*+{\cdots\, ,};
(-30,0)*+{\bullet}="V-3"; (-20,0)*+{\bullet}="V-2";
(-10,0)*+{\bullet}="V-1"; (0,0)*+{\bullet}="V0";
(10,0)*+{\bullet}="V1"; 
 "V-3";"V-2"**\crv{(-30,0)&(-20,0)};
 ?>*\dir{>}\POS?(.5)*+!D{};
 "V-2";"V-1"**\crv{(-20,0)&(-10,0)};
 ?>*\dir{>}\POS?(.5)*+!D{};
 "V-1";"V0"**\crv{(-10,0)&(0,0)};
 ?>*\dir{>}\POS?(.5)*+!D{};
 "V0";"V1"**\crv{(0,0)&(10,0)};
 ?>*\dir{>}\POS?(.5)*+!D{};
 (-25,1.5)*+{a};(-15,1.5)*+{a};(-5,1.5)*+{a};(5,1.5)*+{a};
 (0.1,-2.5)*+{v_{3}};(10.1,-2.5)*+{v_{4}};
 (-9.9,-2.5)*+{v_{2}};
 (-19.9,-2.5)*+{v_{1}};(-29.9,-2.5)*+{v_{0}};
\endxy

\vskip 1pc

\noindent
 it is not hard to see that $\bEz$ consists of all finite sets $F$ 
 with $v_0\notin F$ and all sets of the form $F\cup  \{v_k,v_{k+1}, \dots\}$ for some $k\geq 1$. 
 It is also easy to see that every $A\in \bEz$ containing at least two vertices
 always admits a generalized loop. 
 But there does not exist a loop at  any $A\in \bEz$.  
 Nevertheless we shall see that $C^*(E,\CL,\bEz)$ contains 
 an infinite projection and so the $C^*$-algebra is not AF.
 Let $C^*(E,\CL,\bEz)=C^*(p_A,s_a)$. 
  Then for $B:=r(a)$, we have $r(B,a)\subsetneq B$, 
 and a similar argument as in (ii) shows that the projection $p_B$ is infinite; 
 $p_{r(a)}=s_a p_{r(r(a),a)}s_a^*=s_a p_{r(a^2)}s_a^*\sim p_{r(a^2)}<p_{r(a)}$.
 
 The $C^*$-algebra $C^*(p_A,s_a)$ is unital with the unit $s_a s_a^*$; 
 $(s_a s_a^*)p_A=s_a p_{r(A,a)} s_a^*=p_A$ for all $A\in \bEz$ 
 and $(s_a s_a^*)s_a=s_a= s_a p_{r(a)}=s_a p_{r(a)}(s_a s_a^*) =s_a(s_a s_a^*)$. 
 Also we have  $s_a s_a^*\gneq p_{r(a)}= s_a^* s_a$  since 
 $s_a s_a^*\geq s_a p_{\{v_1\}} s_a^*(\neq 0)$  
 and $(s_a p_{\{v_1\}} s_a^*)p_A= s_a p_{\{v_1\}} p_{r(A,a)}s_a^*
 =s_a p_{\{v_1\}\cap r(A,a)}s_a^*=0$ 
 because $\{v_1\}\cap r(A,a)=\emptyset$ for all $A\in \bEz$. 
 Moreover every projection $p_A$ belongs to the $*$-algebra generated by $s_a$. 
 Therefore $C^*(E,\CL,\bEz)$ is the universal $C^*$-algebra 
 generated by a proper coisometry $s_a$, and thus $C^*(E,\CL,\bEz)$ is the Toeplitz algebra.  
 The ideal $I_{\{v_1\}}$ generated by the projection 
 $p_{\{v_1\}}$ is in fact isomorphic to the the $C^*$-algebra of 
 compact operators on an infinite dimensional separable Hilbert space  
 as $I_{\{v_1\}}=\overline{\rm span}\{ s_a^m p_{\{v_i\}} (s_a^*)^n 
 :\, m,n\geq 0\text { and } i\geq 1\}$ (see (\ref{eqn-ideal-elements})).
 The quotient algebra $C^*(E,\CL,\bEz)/I_{\{v_1\}}$ is isomorphic to $C(\mathbb T)$.
\end{ex}
 
\vskip 1pc

For a labeled graph $(E,\CL_E)$, $v\sim w$ if and only if $v\sim_l w$ for all 
$l\geq 1$ defines an equivalence relation on $E^0$. 
We denote the equivalence class of $v\in E^0$ by $[v]_\infty$.
If $(E,\CL_E)$ has no sinks or sources, 
there exists a labeled graph $(F,\CL_F)$  called the {\it merged labeled graph} 
of $(E,\CL_E)$ with vertices $F^0:=\{[v]_\infty:\, v\in E^0\}$ 
and edges $F^1:=\{e_\ld:\, \ld\in E^1\}$, where  
$e_\ld$ is a path with $s_F(e_\ld)=[s(\ld)]_\infty$, $r_F(e_\ld)=[r(\ld)]_\infty$, 
and  $\CL_F(e_\ld)=\CL_E(\ld)$. 
The range of $a\in \CL_F(F^1)$ is defined by 
$r_F(a)=\{r_F(e_\ld): \CL_F(e_\ld)=a\}$. 
Here we use notation $r_F$ to denote both the range map of paths in $F^*$ and 
of labeled paths in $\CL_F^*(F)$. 
 It is known in \cite[Theorem 6.10]{JKP} 
 that if $[v]_\infty\in \bEz$ for all $v\in  E^0$, 
 then $\{ [v]_\infty\}\in\overline{\CF^0}$ for all $[v]_\infty\in F^0$ and 
 moreover $C^*(E,\CL_E,\bEz)\cong C^*(F,\CL_F,\overline{\CF^0})$. 
  Even when $(E,\CL_E)$ has  sinks or sources, we can obtain 
 $C^*(E,\CL_E,\bEz)\cong C^*(F,\CL_F,\overline{\CF^0})$ 
 whenever $[v]_\infty\in \bEz$ for all $v\in  E^0$ 
 without significant modification of the proof of \cite[Theorem 6.10]{JKP}.

The following proposition is a slightly generalized version of 
the result well known for graph $C^*$-algebras. 
Actually in case $\CL$ is the trivial labeling, 
$C^*(E,\CL,\bEz)=C^*(E)$ and the minimal sets in $\bEz$ are 
the single vertex sets $\{v\}$, $v\in E^0$. 

\vskip 1pc 

\begin{prop}\label{prop F-isom} Let $(E,\CL)$ be a row-finite labeled graph 
with no sinks or sources such that 
 every generalized vertex  is a finite union of minimal sets in $\bEz$.  
 Then $C^*(E,\CL,\bEz)$ is AF if and only if 
 no minimal set of $\bEz$ admits a loop. 
\end{prop}

\begin{proof} Let $(F,\CL_F)$ be the merged labeled graph of $(E,\CL)$. 
We first show that  $C^*(E,\CL,\bEz)$ is isomorphic to the graph $C^*$-algebra $C^*(F)$. 

Our assumption implies $[v]_\infty\in \bEz$ for all $v\in E^0$, so 
 $\{[v]_\infty\}\in \bFz$ for all $v\in E^0$  and 
 $C^*(E,\CL,\bEz)$ is isomorphic to $C^*(F, \CL_F, \bFz)$ 
(\cite[Theorem 6.10]{JKP}).  
For each $a \in \CL(E^1)$, its range $r(a)$ can be written 
as the union $r(a) = \cup_{i=1}^{n} [w_i]_{l_i}$ of  
finitely many minimal sets $[w_i]_{l_i}$ by the assumption, 
but minimality of each  $[w_i]_{l_i}$ implies that 
$[w_i]_{l_i}=[w_i]_\infty$ for $1\leq i\leq n$.   
Hence $r_F(a)=[{r(a)}]_\infty := \{[w]_\infty: w\in r(a)\} 
=\{[w_1]_\infty, \dots ,[w_n]_\infty\}$ is finite for each $a\in \CA$. 
But from the construction (\cite[Definition 6.1]{JKP}), 
the merged labeled graph $(F,\CL_F)$ is left-resolving. 
Thus the finiteness of each range set $r_F(a)$ implies that 
$(F, \CL_F)$ is label-finite. 
Then by  \cite[Theorem 6.6]{BP1}, we have $C^*(F, \CL_F, \bFz) \cong C^*(F)$. 

Suppose that there is no loop at any minimal set 
$[v]_\infty$ in $\bEz$. 
Since $\CL_E([v]_\infty E^k v')=\CL_F([v]_\infty F^k [v]_\infty)$ 
for all $v'\in [v]_\infty$ and $k\geq 1$ (\cite[Lemma 6.7]{JKP}), 
if $F$ has a loop $\af$ at a vertex $[v]_\infty\in F^0$, 
$\af\in \CL_E([v]_\infty E^k v')$ for all $v'\in [v]_\infty$. 
This means that $[v]_\infty(\in \bEz)$ satisfies 
$[v]_\infty\subset r([v]_\infty, \af)$, a contradiction. 
Hence $F$ has no loops and the $C^*$-algebra  $C^*(F)$ is AF. 
The converse was proved in Theorem~\ref{thm-loopAF}.
\end{proof}

\vskip 1pc

\begin{ex}\label{ex-noinfinitepath} In the following labeled graph $(E,\CL)$
\vskip .5pc

\hskip 1pc \xy /r0.38pc/:(-33,0)*+{\cdots};
(-30,0)*+{\bullet}="V-3"; (-20,0)*+{\bullet}="V-2";
(-10,0)*+{\bullet}="V-1"; (0,0)*+{\bullet}="V0";
(10,0)*+{\bullet}="V1"; (20,0)*+{\bullet}="V2";
(30,0)*+{\bullet}="V3";
(-33,8)*+{\cdots};
(-30,8)*+{\bullet}="V-31"; (-20,8)*+{\bullet}="V-21";
(-10,8)*+{\bullet}="V-11"; (0,8)*+{\bullet}="V01";
(10,8)*+{\bullet}="V11"; (20,8)*+{\bullet}="V21";
(30,8)*+{\bullet}="V31";
(33,8)*+{\cdots};
(-33,16)*+{\cdots};
(-30,16)*+{\bullet}="V-32"; (-20,16)*+{\bullet}="V-22";
(-10,16)*+{\bullet}="V-12"; (0,16)*+{\bullet}="V02";
(10,16)*+{\bullet}="V12"; (20,16)*+{\bullet}="V22";
(30,16)*+{\bullet}="V32";
(33,16)*+{\cdots};
(0,19)*+{\vdots};
(-10,19)*+{\vdots};(10,19)*+{\vdots};
(-20,19)*+{\vdots};(20,19)*+{\vdots};
(-30,19)*+{\vdots};(30,19)*+{\vdots};
 "V-32";"V-31"**\crv{(-30,16)&(-30,8)};
 ?>*\dir{>}\POS?(.5)*+!D{}; "V-31";"V-3"**\crv{(-30,8)&(-30,0)};
 ?>*\dir{>}\POS?(.5)*+!D{};
 "V-22";"V-21"**\crv{(-20,16)&(-20,8)};
 ?>*\dir{>}\POS?(.5)*+!D{}; "V-21";"V-2"**\crv{(-20,8)&(-20,0)};
 ?>*\dir{>}\POS?(.5)*+!D{};
  "V-12";"V-11"**\crv{(-10,16)&(-10,8)};
 ?>*\dir{>}\POS?(.5)*+!D{}; "V-11";"V-1"**\crv{(-10,8)&(-10,0)};
 ?>*\dir{>}\POS?(.5)*+!D{};
 "V02";"V01"**\crv{(0,16)&(0,8)};
 ?>*\dir{>}\POS?(.5)*+!D{}; "V01";"V0"**\crv{(0,8)&(0,0)};
 ?>*\dir{>}\POS?(.5)*+!D{};
  "V12";"V11"**\crv{(10,16)&(10,8)};
 ?>*\dir{>}\POS?(.5)*+!D{}; "V11";"V1"**\crv{(10,8)&(10,0)};
 ?>*\dir{>}\POS?(.5)*+!D{};
  "V22";"V21"**\crv{(20,16)&(20,8)};
 ?>*\dir{>}\POS?(.5)*+!D{}; "V21";"V2"**\crv{(20,8)&(20,0)};
 ?>*\dir{>}\POS?(.5)*+!D{};
  "V32";"V31"**\crv{(30,16)&(30,8)};
 ?>*\dir{>}\POS?(.5)*+!D{}; "V31";"V3"**\crv{(30,8)&(30,0)};
 ?>*\dir{>}\POS?(.5)*+!D{};
  "V-3";"V-2"**\crv{(-30,0)&(-20,0)};
 ?>*\dir{>}\POS?(.5)*+!D{};
  "V-2";"V-1"**\crv{(-20,0)&(-10,0)};
 ?>*\dir{>}\POS?(.5)*+!D{};
 "V-1";"V0"**\crv{(-10,0)&(0,0)};
 ?>*\dir{>}\POS?(.5)*+!D{};
 "V0";"V1"**\crv{(0,0)&(10,0)};
 ?>*\dir{>}\POS?(.5)*+!D{};
 "V1";"V2"**\crv{(10,0)&(20,0)};
 ?>*\dir{>}\POS?(.5)*+!D{};
 "V2";"V3"**\crv{(20,0)&(30,0)};
 ?>*\dir{>}\POS?(.5)*+!D{};
(33,0)*+{\cdots ,};
 (-25,1)*+{a};(-15,1)*+{a};(-5,1)*+{a};(5,1)*+{a};
 (15,1)*+{a};(25,1)*+{a};
 (0.1,-2.5)*+{v_0};(10.1,-2.5)*+{v_1};
 (2,8)*+{u_0};(12,8)*+{u_1};
 (22,8)*+{u_2};(-7,8)*+{u_{-1}};(-17,8)*+{u_{-2}};(-27,8)*+{u_{-3}};
 (-9.9,-2.5)*+{v_{-1}};
 (-19.9,-2.5)*+{v_{-2}};(-29.9,-2.5)*+{v_{-3}};
 (20.1,-2.5)*+{v_{2}};(30.1,-2.5)*+{v_{3}};
 (-1.1,4)*+{b};(-1.1,12)*+{b};
 (-21.1,4)*+{b};(-21.1,12)*+{b};
 (18.9,4)*+{b};(18.9,12)*+{b};
 (-11.1,4)*+{c};(-11.1,12)*+{c};
 (8.9,4)*+{c};(8.9,12)*+{c};
 (28.9,4)*+{c};(28.9,12)*+{c};
\endxy

\vskip 1pc
\noindent the path $\af:=a^2$ is a loop at 
$\{v_{2k}:\, k\in \mathbb Z\}$ and also at $\{v_{2k+1}:\, k\in \mathbb Z\}$. 
By Theorem~\ref{thm-loopAF}, the $C^*$-algebra $C^*(E,\CL,\bEz)$ is not AF. 
In fact, $C^*(E,\CL,\bEz)$ is isomorphic to the graph algebra 
$C^*(F)$, where $F$ is the underlying graph of the merged labeled graph 
$(F,\CL_F)$ of $(E,\CL)$ by Proposition~\ref{prop F-isom};  
\vskip 1pc
\hskip 3pc
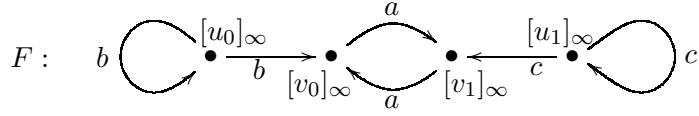
\xy
/r0.38pc/:
(0,0)*+{\bullet}="u";
 (10,0)*+{\bullet}="V0";
  (20,0)*+{\bullet}="V1";
   (30,0)*+{\bullet}="w";
 "u";"u"**\crv{(0,0)&(-4.5,5)&(-9,0)&(-4.5,-5)&(0,0)};?>*\dir{>}\POS?(.5)*+!D{};
 "w";"w"**\crv{(31,0)&(35.5,5)&(40,0)&(35.5,-5)&(31,0)};?>*\dir{>}\POS?(.5)*+!D{};
 "u";"V0"**\crv{(0,0)&(9,0)};?>*\dir{>}\POS?(.5)*+!D{};
 "w";"V1"**\crv{(29,0)&(21,0)};?>*\dir{>}\POS?(.5)*+!D{};
  "V0";"V1"**\crv{(10.5,0)&(15,5.5)&(20,0)};?>*\dir{>}\POS?(.5)*+!D{};
   "V1";"V0"**\crv{(20,0)&(15,-5.5)&(10.5,0)};?>*\dir{>}\POS?(.5)*+!D{};
 (-9,0)*+{b},(40,0)*+{c\,}, (15,4)*+{a},(15,-3.7)*+{a},(4,-1)*+{b},(27,-1)*+{c},
 (1.9,2)*+{[u_0]_\infty},(9,-2.3)*+{[v_0]_\infty},
 (22,-2.3)*+{[v_1]_\infty},(29,2)*+{[u_1]_\infty},
 (-15,0)*+{F:},
\endxy
\end{ex}
 
\vskip 1pc

\begin{ex}\label{ex-thm} The following labeled graph $(E,\CL)$ 
does not have any infinite paths, but it has a repeatable path $a$. 

\vskip .5pc

\hskip 7pc 
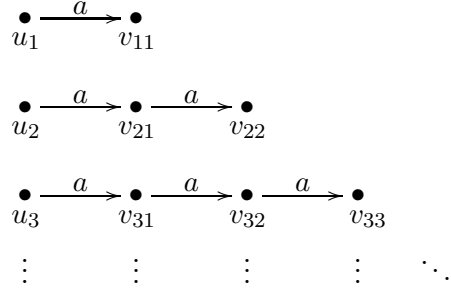
\xy /r0.35pc/: 
(-30,0)*+{\bullet}="U1"; (-20,0)*+{\bullet}="V11";
(-30,-8)*+{\bullet}="U2"; (-20,-8)*+{\bullet}="V21";
(-10,-8)*+{\bullet}="V22";
(-30,-16)*+{\bullet}="U3"; (-20,-16)*+{\bullet}="V31";
(-10,-16)*+{\bullet}="V32"; (0,-16)*+{\bullet}="V33";
(-30,-22)*+{\vdots}; 
(-20,-22)*+{\vdots};
(-10,-22)*+{\vdots};
(0,-22)*+{\vdots};
(7,-22)*+{\ddots};
 "U1";"V11"**\crv{(-30,0)&(-20,0)};?>*\dir{>}\POS?(.5)*+!D{}; 
 "U2";"V21"**\crv{(-30,-8)&(-20,-8)}; ?>*\dir{>}\POS?(.5)*+!D{};
 "V21";"V22"**\crv{(-20,-8)&(-10,-8)};?>*\dir{>}\POS?(.5)*+!D{}; 
 "U3";"V31"**\crv{(-30,-16)&(-20,-16)};?>*\dir{>}\POS?(.5)*+!D{};
 "V31";"V32"**\crv{(-20,-16)&(-10,-16)};?>*\dir{>}\POS?(.5)*+!D{}; 
 "V32";"V33"**\crv{(-10,-16)&(0,-16)};?>*\dir{>}\POS?(.5)*+!D{};
 (-25,1)*+{a};(-25,-7)*+{a};(-25,-15)*+{a};  
 (-15,-7)*+{a};(-15,-15)*+{a};(-5,-15)*+{a}; 
 (-29.9,-2)*+{u_1};(-19.9,-2)*+{v_{11}};
 (-29.9,-10)*+{u_2};(-19.9,-10)*+{v_{21}};(-9.9,-10)*+{v_{22}};
 (-29.9,-18)*+{u_3};(-19.9,-18)*+{v_{31}};(-9.9,-18)*+{v_{32}};(0.9,-18)*+{v_{33}};   
\endxy

\vskip 1pc

\noindent   
Note that each finite path  $a^n$ is not a loop  at any $A\in \bEz$ but it is a 
generalized loop at $r(a^k)$ for all $k\geq 1$.  
$(E, \CL, \bEz)$ has the generalized vertices as follows: 
\begin{align*}
[v_{ij}]_k & =\left\{
       \begin{array}{ll}
          r(a^k), & \hbox{if\ } 1\leq k\leq j \\
           r(a^j)\setminus r(a^{j+1}), & \hbox{if\ }1\leq j < k
       \end{array}
             \right. \\
([v_{ij}]_k)_\sk & =\left\{
       \begin{array}{ll}
          \{v_{mm}:\, m\geq k\, \}, & \hbox{if\ } 1\leq k\leq j \\
          \{v_{jj}\}, & \hbox{if\ }1\leq j < k
       \end{array}
             \right.\\ 
[v_{ij}]_k\setminus ([v_{ij}]_k)_\sk  & =\left\{
       \begin{array}{ll}
          \{v_{mn}: \, m>n\geq k\, \}, & \hbox{if\ } 1\leq k\leq j \\
           \{v_{mj}:\, m\geq j\, \}, & \hbox{if\ }1\leq j < k, 
       \end{array}
             \right. 
\end{align*}   
\noindent 
and every  $A\in \bEz$ is a finite union of these sets. 
    
Let $J$  be the ideal of $C^*(E,\CL, \bEz)=C^*(p_A,s_a)$ generated by 
the projection $p_{[v_{11}]_2}$. 
Then (\ref{eqn-ideal-elements}) shows that
$$J= \overline{\rm span}\{s_a^m p_{B} {s_a^*}^n : 
\ B\in  [v_{kk}]_{k+1} \cap \bEz, \ m,n \geq 0,\ k\geq 1\}.$$ 
From 
$p_{r(a)}-p_{r(a^2)}=p_{r(a)\setminus r(a^2)}=p_{[v_{11}]_2}\in J$, we have
$$s_a+J=s_ap_{r(a)}+J=p_{r(a)}s_a+J.$$ 
Thus  $s_a p_{r(a)}+J$ is a unitary of the unital hereditary subalgebra  
(with unit $p_{r(a)}+J$) of the quotient algebra $C^*(E,\CL,\bEz)/J$. 
The ideal $J$ is obviously invariant under the gauge action 
$\gm: \mathbb T\to {\rm Aut}(C^*(E,\CL,\bEz))$. 
Hence there exists an induced action $\gm:\mathbb T \to {\rm Aut}(C^*(E,\CL,\bEz)/J)$ 
such that $\gm_z(s_a p_{r(a)}+J)=z (s_a p_{r(a)}+J)$ for $z\in \mathbb T$. 
Thus the unitary $s_a p_{r(a)}+K$ does not belong to 
the unitary path connected component of the unit 
of the hereditary subalgebra of  $C^*(E,\CL,\bEz)/J$, which implies 
as in the proof of Theorem~\ref{thm-loopAF}
that $C^*(E,\CL,\bEz)/J$ and hence $C^*(E,\CL,\bEz)$ is not AF.   
\end{ex}

\vskip 1pc

\begin{notation} 
If $x_i\in A_iE^{\leq K} A_{i+1}$  
is a path with $\af_i =\CL(x_i)$ for $i=1,\dots,n$ such that 
$x_1\cdots x_n\in A_1 E^{\leq K}\cdots E^{\leq K} A_{n+1}$, then 
we set 
\begin{align*}
{\bar r}(A_{1}\af_1 A_{2}) &:=r(A_{1},\af_1)\cap A_{2}\\
{\bar r}(A_{1}\af_1 A_{2}\af_2 A_3) & :=r({\bar r}(A_{1}\af_1 A_{2}) , \af_2)\cap A_3 = 
 r(r(A_{1},\af_1)\cap A_{2},\af_2)\cap A_3,
\end{align*}  and so on, thus  for $3\leq i\leq n+1$,  
$${\bar r}(A_{1}\af_1 A_{2}\cdots \af_{i-1} A_{i}) := 
r({\bar r}(A_{1}\af_1 A_{2}\cdots A_{{i-1}}),\af_{i-1})\cap A_i.$$  
Note that ${\bar r}(A_{1}\af_1 A_{2}\cdots \af_{i-1} A_{i})$ belongs to 
$\bEz$ whenever $A_j\in \bEz$ for $1\leq j\leq i$. 
The notation 
${\bar r}(A_{1}E^{\leq K} A_{2} \cdots E^{\leq K}A_{n+1})$ will then be used 
for the  collection of all sets 
${\bar r}(A_{1}\af_1 A_{2}\cdots \af_{n-1} A_{n+1})$ for   
$\af_1\cdots\af_{n} \in \CL(A_{1}E^{\leq K}A_2 \cdots E^{\leq K}A_{n+1})$.
\end{notation} 
\vskip 1pc

\begin{thm}\label{thm-AF} Let $(E,\CL,\bEz)$ be a labeled space such that  
for every finite subset  $\{A_1, \dots, A_N\}$ of $\bEz$ and every $K\geq 1$, 
there exists an $m_0\geq 1$ for which 
$$ A_{i_1} E^{\leq K} A_{i_2} E^{\leq K} A_{i_3}\cdots E^{\leq K} A_{i_n} =\emptyset$$ 
for all $n> m_0$ and $1\leq i_j\leq N$. 
Then $C^*(E,\CL,\bEz)$ is an AF algebra. 
\end{thm}

\begin{proof} 
Let $F:=\{s_{\af_i} p_{A_i} s_{\bt_i}^*:\, A_i\subset r(\af_i)\cap r(\bt_i),\, i=1, \dots, N\}$ be 
a finite set in the $C^*$-algebra $C^*(E,\CL,\bEz)=C^*(s_a,p_A)$ 
with $F=F^*$. 
We shall show that $F$ generates a finite dimensional $C^*$-algebra.  
Set $K:=\max\{|\af_i|, |\bt_i|:\, i=1, \dots N\}$. 
By Remark~\ref{remark-elements}(i), we have
$$(s_{\af_i} p_{A_i} s_{\bt_i}^*)(s_{\af_j} p_{A_j} s_{\bt_j}^*)=\left\{
                      \begin{array}{ll}
                        s_{\af_i\gm'}p_{r(A_i,\gm')\cap A_j} s_{\bt_j}^*, & \hbox{if\ } \af_j=\bt_i\gm' \\
                        s_{\af_i}p_{A_i\cap r(A_j,\bt')} s_{\bt_j\bt'}^*, & \hbox{if\ } \bt_i=\af_j\bt'\\
                        s_{\af_i} p_{A_i\cap A_j}s_{\bt_j}^*, & \hbox{if\ } \bt_i=\af_j\\
                        0, & \hbox{otherwise,}
                      \end{array}
                    \right.
$$ 
and so if, for example,  $\af_j=\bt_i\gm'$ and  $\af_k=\bt_j\gm''$, we get  
\begin{align*} 
(s_{\af_i} p_{A_i} s_{\bt_i}^*)(s_{\af_j} p_{A_j} s_{\bt_j}^*)(s_{\af_k} p_{A_k} s_{\bt_k}^*) 
  & =  (s_{\af_i\gm'}p_{r(A_i,\gm')\cap A_j} s_{\bt_j}^*)(s_{\af_k} p_{A_k} s_{\bt_k}^*)\\
  &=  s_{\af_i\gm'\gm''}p_{r(r(A_i,\gm')\cap A_j,\gm'')\cap A_k} s_{\bt_k}^*.
\end{align*}
Here note that 
$ \gm'\gm''$ belongs to $\CL(A_i E^{|\gm'|}A_j E^{|\gm''|}A_k)$ and 
the set $ r(r(A_i,\gm')\cap A_j,\gm'')\cap A_k$ is equal to 
${\bar r}(A_i\gm'A_j\gm''A_k)$. 
Continuing a similar computation once more, for example with $\bt_k=\af_l \bt'$, we have 
\begin{align*} 
   & \ (s_{\af_i} p_{A_i} s_{\bt_i}^*)(s_{\af_j} p_{A_j} s_{\bt_j}^*)(s_{\af_k} p_{A_k} s_{\bt_k}^*)
(s_{\af_l} p_{A_l} s_{\bt_l}^*)  \\
 = & \  (s_{\af_i\gm'\gm''}p_{r(r(A_i,\gm')\cap A_j,\gm'')\cap A_k} s_{\bt_k}^*)
  (s_{\af_l} p_{A_l} s_{\bt_l}^*)\\
 = & \  s_{\af_i\gm'\gm''}p_{r(r(A_i,\gm')\cap A_j,\gm'')\cap A_k\cap r(A_l,\bt')} s_{\bt_l\bt'}^* 
\end{align*}
which is nonzero only when   
$\gm'\gm''\in \CL(A_i E^{|\gm'|}A_j E^{|\gm''|}A_k)$ and 
$\bt'\in \CL(A_l E^{|\bt'|}A_k)$.  
If this is the case, we have    
$$ s_{\af_i\gm'\gm''}p_{r(r(A_i,\gm')\cap A_j,\gm'')\cap A_k\cap r(A_l,\bt')} s_{\bt_l\bt'}^*  
= s_{\af_i\gm'\gm''}p_{{\bar r}(A_i\gm'A_j\gm''A_k)\cap {\bar r}(A_l\bt'A_k)} s_{\bt_l\bt'}^* 
$$ as before. 
 Repeating the process of multiplying any finite elements from the set $F$ 
actually produces an element of the form $s_{\af_i\mu} p_A s_{\bt_j\nu}^*$, where 
$A$ is a finite intersection of sets in 
$$ A(F):= \cup_{\substack{n\geq 1\\ 1\leq i_j\leq N}}\, {\bar r}\big( A_{i_1} E^{\leq K} A_{i_2} 
\cdots  E^{\leq K}A_{i_n} \big)$$ 
and $\mu$ and $\nu$ are paths in 
$$\CL(F):=\cup_{\substack{n\geq 1\\ 1\leq i_j\leq N}}\,
\CL(A_{i_1} E^{\leq K} A_{i_2} \cdots E^{\leq K}A_{i_n}).$$   
By our assumption, we find an $m_0\geq 1$ such that  
$\CL(A_{i_1} E^{\leq K} A_{i_2} \cdots E^{\leq K}A_{i_n})=\emptyset$ for all $n>m_0$,  
so that  $\CL(F)$ turns out to be a finite set  
since our labeled space is always assumed receiver set-finite. 
Then the finiteness of the set $A(F)$ is immediate, and so       
we conclude that $F$ generates the finite dimensional $*$-algebra;    
$$\overline{\rm span}\big\{ s_{\af_i\mu}p_A s_{\bt_j\nu}^*:\, A=\cap B_k,\ 
B_k\in A(F), \ \mu,\nu\in \CL(F),\, 1\leq i,j\leq N\, \big\}.$$
\end{proof}

\vskip 1pc
In the following example, we see that the condition that 
$(E,\CL,\bEz)$ has no repeatable paths is not a sufficient condition for 
$C^*(E,\CL,\bEz)$ to be AF.

\vskip 1pc 

\begin{ex}\label{ex-Morse} Consider the following labeled graph 
$(E,\CL)$:
\vskip 1pc

 \xy /r0.34pc/:(-45,0)*+{\cdots};(45,0)*+{\cdots ,};
(-40,0)*+{\bullet}="V-4";
(-30,0)*+{\bullet}="V-3";
(-20,0)*+{\bullet}="V-2";
(-10,0)*+{\bullet}="V-1"; (0,0)*+{\bullet}="V0";
(10,0)*+{\bullet}="V1"; (20,0)*+{\bullet}="V2";
(30,0)*+{\bullet}="V3";
(40,0)*+{\bullet}="V4";
 "V-4";"V-3"**\crv{(-40,0)&(-30,0)};
 ?>*\dir{>}\POS?(.5)*+!D{};
 "V-3";"V-2"**\crv{(-30,0)&(-20,0)};
 ?>*\dir{>}\POS?(.5)*+!D{};
 "V-2";"V-1"**\crv{(-20,0)&(-10,0)};
 ?>*\dir{>}\POS?(.5)*+!D{};
 "V-1";"V0"**\crv{(-10,0)&(0,0)};
 ?>*\dir{>}\POS?(.5)*+!D{};
 "V0";"V1"**\crv{(0,0)&(10,0)};
 ?>*\dir{>}\POS?(.5)*+!D{};
 "V1";"V2"**\crv{(10,0)&(20,0)};
 ?>*\dir{>}\POS?(.5)*+!D{};
 "V2";"V3"**\crv{(20,0)&(30,0)};
 ?>*\dir{>}\POS?(.5)*+!D{};
 "V3";"V4"**\crv{(30,0)&(40,0)};
 ?>*\dir{>}\POS?(.5)*+!D{};
 (-35,1.5)*+{0};(-25,1.5)*+{1};
 (-15,1.5)*+{1};(-5,1.5)*+{0};(5,1.5)*+{x_0=0};
 (15,1.5)*+{x_1=1};(25,1.5)*+{1};(35,1.5)*+{0};
 (0.1,-2.5)*+{v_0};(10.1,-2.5)*+{v_1};
 (-9.9,-2.5)*+{v_{-1}};
 (-19.9,-2.5)*+{v_{-2}};
 (-29.9,-2.5)*+{v_{-3}};
 (-39.9,-2.5)*+{v_{-4}}; 
 (20.1,-2.5)*+{v_{2}};
 (30.1,-2.5)*+{v_{3}};
 (40.1,-2.5)*+{v_{4}}; 
\endxy 

\vskip 1pc
\noindent 
where the $\{0,1\}$ sequence 
is the {\it  Morse sequence} 
$$x=\cdots x_{-2} x_{-1} x_0 x_1 x_2\cdots $$ 
given by 
$x_0=0$, $x_1=\bar 0:=1$ ($\bar 1:=0$), 
$x_{[0,3]}:=x_0 x_1 \bar x_0 \bar x_1 =0110$, 
$x_{[0,7]}:=x_0 x_1 x_2 x_3\bar x_0 \bar x_1 \bar x_2 \bar x_3 = 01101001$ 
and so on, and then $x_{-i}:=x_{i-1}$ for $i\geq 1$. 
It is known that 
$x$ contains no block (no finite subsequence) of the form 
$\bt\bt \bt_1$ for $\bt=\bt_1\cdots \bt_{|\bt|}\in \CL^*(E)$. 
(See \cite{GH} for the Morse sequence.)
Thus $(E,\CL,\bEz)$ has no repeatable paths satisfying 
(b) in Remark~\ref{rmk-AF-eq}. 
But, the set $A:=r(0)$, with $K:=2$, satisfies  
$$A_{i_1}E^{\leq 2}A_{i_2} \cdots E^{\leq 2} A_{i_n}\neq \emptyset$$ 
for all $n\geq 1$, where 
$A_{i_j}=A$, $j\geq 1$.  
This is because the block $111$ does not appear in the sequence $x$. 
Thus $(E,\CL,\bEz)$ does not 
meet the condition (a) in Remark~\ref{rmk-AF-eq}. 
To see that $C^*(E,\CL,\bEz)$ 
(equivalently, $M_2\otimes C^*(E,\CL,\bEz)$) is not AF, 
it is enough to show that 
$M_2\otimes C^*(E,\CL,\bEz)$ contains a unitary $U$ such that 
$(id_{M_2}\otimes\gm)_z(U)=zU$ for all $z\in \mathbb T$, 
where $\gm$ is the gauge action of $\mathbb T$ on  
$C^*(E,\CL,\bEz)=C^*(s_i,p_A)$ 
(\cite[Proposition 3.9]{ES}). 
Actually one can easily check that  the unitary $U=(u_{ij})$, with entries 
$u_{ij}=\dt_{ij}\,s_0 + (1-\dt_{ij})s_1$, is a desired one.   
\end{ex}

\vskip 1pc

Now  we prove a partial result about the implication 
(c) $\Rightarrow$ (b) of  Remark~\ref{rmk-AF-eq}. 
For a   $C^*$-algebra $C^*(E,\CL,\bEz)=C^*(s_a,p_A)$ and 
a set $A\in \bEz$, we denote by $I_A$ the ideal of  $C^*(E,\CL,\bEz)$ 
generated by the projection $p_A$ as before. 

\vskip 1pc

\begin{lem}\label{lem-ideal} 
Let $C^*(E,\CL,\bEz)=C^*(s_a,p_A)$ be the 
$C^*$-algebra of a labeled graph $(E,\CL)$ with no sinks or sources.  
For $A, B\in \bEz$,  we have 
$p_A\in I_B$ if and only if 
there exist an $N\geq 1$ and finitely many paths $\{\mu_i\}_{i=1}^n$   
in $\CL(BE^{\geq 0})$ such that 
$$\cup_{|\bt|=N} \, r(A,\bt)\subset \cup_{i=1}^n r(B,\mu_i).$$
\end{lem}
\begin{proof} 
If $p_A\in I_B$, we can approximate $p_{A}$, within  a small enough $\varepsilon>0$,  
by an element  $\sum_{i=1}^n c_i s_{\bt_i}p_{B_i\cap r(B,\mu_i)} s_{\gm_i}^*$ of $I_B$, 
where $c_i\in \mathbb C$, $\bt_i, \gm_i\in \CL(A  E^{\geq 0})$, $B_i\in \bEz$, and
 $\mu_i\in  \CL(B E^{\geq 0})$ for $1\leq i\leq n$ (see (\ref{eqn-ideal-elements})). 
 We assume $(\bt_i, \mu_i,\gm_i)\neq (\bt_j,\mu_j,\gm_j)$ if $i\neq j$.
Considering the image of 
$X:= p_{A}-\sum_{i=1}^n c_i s_{\bt_i}p_{B_i\cap r(B,\mu_i)} s_{\gm_i}^*$ 
under the conditional expectation onto the AF core 
(the fixed point algebra of the gauge action), 
we may assume that $|\bt_i|=|\gm_i|$, $1\leq i\leq n$, 
since $p_A$ is in the core. 
Moreover, since $(E,\CL)$ has no sinks, 
we can also assume that $|\bt_i|=|\bt_1|$ for all $i$. 
Put $N :=|\bt_i|$, $1\leq i\leq n$.  
From  
$p_{A}=\sum_{|\bt|=N} s_\bt p_{r(A,\bt)} s_\bt^*$,  
 we have 
$$\|X\|=\big\| \sum_{|\bt|=N} s_\bt p_{r(A,\bt)} s_\bt^*   
- \sum_{i=1}^n c_i s_{\bt_i}p_{B_i\cap r(B,\mu_i)} s_{\gm_i}^*\big\|<\varepsilon. $$ 
If $r(A,\bt)\not\subset  \cup_{i=1}^n r(B,\mu_i)$ for some $\bt\in \CL(AE^N)$, 
that is, $A':=r(A,\bt)\setminus \cup_{i=1}^n r(B,\mu_i)\neq \emptyset$,  
one obtains a contradiction, 
$\varepsilon > \|p_{A'}(s_\bt^* Xs_\bt) p_{A'}\|=\|p_{A'}\|=1$. 

For the reverse inclusion, it is enough to note that  
$p_{\cup_{i=1}^n r(B,\mu_i)}\in I_B$ (see \cite[Lemma 3.5]{JKP}). 
\end{proof}

\vskip 1pc
If $\af$ is a repeatable path in a directed graph $E$, then $\af$ 
is a loop with the range $r(\af)$ consisting of 
a single vertex and every repetition $\af^n$ also has 
the same range as $\af$, $r(\af^m)=r(\af)$, $m\geq 1$. 
The projection $p_{r(\af)\setminus r(\af^m)}$ is then equal to 
$0$ in the  $C^*$-algebra $C^*(E,\CL_{id},\bEz)$, and so the (zero) ideal 
generated by the projection $p_{r(\af)\setminus r(\af^m)}$ 
can not have the nonzero projection $p_{r(\af)}$. 
In this case, we already know that $C^*(E)=C^*(E,\CL_{id},\bEz)$ is not AF.

\vskip 1pc 

\begin{thm}\label{thm-repeatablepath} 
Let $C^*(E,\CL,\bEz)=C^*(s_a,p_A)$ be the 
$C^*$-algebra of a labeled graph $(E,\CL)$ with no sinks or sources. 
Let $C^*(E,\CL,\bEz)$ have a repeatable path $\af\in \CL^*(E)$.   
If $p_{r(\af^m)}$ does not belong to the ideal generated by a projection 
$p_{r(\af^m)\setminus r(\af^{m+1})}$ for some $m\geq 1$, 
$C^*(E,\CL,\bEz)$ is not AF. 
\end{thm}

\begin{proof} Let $A_m:=r(\af^m)\setminus r(\af^{m+1})$. 
Then $\{I_{A_m}\}_{m=1}^\infty$ is a decreasing sequence of ideals because 
the generator  $ p_{r(\af^{m+1})\setminus r(\af^{m+2})}$ of $I_{A_{m+1}}$ 
belongs to $I_{A_m}$; 
$$ p_{r(\af^{m+1})\setminus r(\af^{m+2})}
= s_\af^*s_\af\, p_{r(r(\af^{m})\setminus r(\af^{m+1}),\af)}
= s_\af^*p_{r(\af^m)\setminus r(\af^{m+1})}s_\af\in I_{A_m}.$$ 
We first show the following claim.
\vskip .5pc

\noindent
{\bf Claim}: If $p_{r(\af)}$ does not belong to the ideal   
generated by $p_{r(\af)\setminus r(\af^2)}$, then 
the $C^*$-algebra $C^*(E,\CL,\bEz)$ is not AF. 

\vskip .5pc
\noindent To prove the claim, it is enough to show that 
the quotient algebra $C^*(E,\CL,\bEz)/I_{A_1}$ is not AF. 
Note that 
$p_{r(\af)} + I_{A_1} =p_{r(\af^2)}+ I_{A_1}$ is 
a nonzero projection in the quotient algebra $C^*(E,\CL,\bEz)/I_{A_1}$ and that  
$$I_{A_1}
=\overline{\rm span}\big\{ s_\bt p_B s_\gm^*:\, \bt,\, \gm\in\CL(E^{\geq 0})
\text{ and } B\in r(\CL(A_1 E^{\geq 0}))\cap \bEz \, \big\}$$   
by (\ref{eqn-ideal-elements}). 
 If $s_\af^* s_\af + I_{A_1}= s_\af s_\af^*+ I_{A_1}$, 
the hereditary subalgebra of 
$C^*(E,\CL,\bEz)/I_{A_1}$ with the unit projection $p_{r(\af)}+ I_{A_1}$ 
is not AF since it contains a unitary $s_\af + I_{A_1}$ satisfying 
$\gm_z(s_\af + I_{A_1})=z^{|\af|}(s_\af + I_{A_1})$ for each $z\in \mathbb C$. 
Thus the hereditary subalgebra (hence $C^*(E,\CL,\bEz))$ is not an AF algebra. 
 (The fact that $s_\af +I_{A_1}$ belongs to the hereditary subalgebra 
 follows from $p_{r(\af)}s_\af+ I_{A_1} =s_\af p_{r(\af^2)}+ I_{A_1}
=s_\af p_{r(\af)}+ I_{A_1}=s_\af +I_{A_1}$.) 
If $s_\af^* s_\af +I_{A_1}\neq s_\af s_\af^*+ I_{A_1}$, then  
$s_\af^* s_\af+ I_{A_1}=p_{r(\af)}+ I_{A_1}\geq s_\af p_{r(\af^2)}s_\af^*+I_{A_1}
=s_\af p_{r(\af)}s_\af^*+I_{A_1}= s_\af s_\af^*+I_{A_1}$ and   
this shows that $s_\af^* s_\af +I_{A_1}\gneq s_\af s_\af^*+I_{A_1}$. 
Thus the projection $s_\af^* s_\af + I_{A_1}$ is infinite, 
and the quotient algebra is not AF as claimed. 
 
Now suppose that  $p_{r(\af^m)}\notin I_{A_m}$ for some $m\geq 2$. 
Since  $\dt:=\af^m$ is a repeatable path, 
by the above claim,  we only need to show 
that $p_{r(\dt)}$ does not belong to the ideal, say $J$, generated by the projection 
$p_{r(\dt)\setminus r(\dt^2)}= p_{r(\af^m)\setminus r(\af^{2m})}$.
For this, assuming $p_{r(\dt)}\in J$ we have from Lemma~\ref{lem-ideal} that 
there exist an $N\geq 1$ and paths $\{\mu_j\}_{j=1}^n$ such that 
$$r(r(\dt), \bt)\subset 
\cup_{i=1}^n r\big(r(\af^m)\setminus r(\af^{2m}),\mu_i\big)$$ 
for all $\bt\in \CL(r(\dt)E^N)$.
Since each set $r(r(\af^m)\setminus r(\af^{2m}),\mu_i)$ coincides with   
$$\cup_{j=0}^{m-1}\, r\big(r(\af^{m+j})\setminus r(\af^{m+j+1}),\,\mu_i\big)
=\cup_{j=0}^{m-1}\, r\big(r(\af^{m})\setminus r(\af^{m+1}),\af^j\mu_i\big) ,$$ 
we can write the set 
$\cup_{i=1}^n r\big(r(\af^m)\setminus r(\af^{2m}),\,\mu_i\big)$ as 
$\cup_{j=1}^{n'} r\big(r(\af^m)\setminus r(\af^{m+1}),\mu'_j)$ for some 
finitely many paths $\mu'_j$ which is of the form $\af^l\mu_i$. 
This  means that
$ p_{r(\af^m)}= p_{r(\dt)} \in I_{A_m}$ again by Lemma~\ref{lem-ideal}, 
which is a contradiction. 
\end{proof}

\vskip 1pc

Assume that $E$ is a graph with no sinks or sources. 
Recall from \cite[Definition 5.1]{BP1} and \cite[Proposition 3.9]{JKP} 
that a labeled space $(E,\CL,\bEz)$ is {\it disagreeable} if 
$[v]_l$ is disagreeable for all $v\in E^0$ and $l\geq 1$; 
a generalized vertex $[v]_l$ is {\it not disagreeable} if and only if 
there is an $N>0$ such that every path $\af\in \CL([v]_l E^{\geq N})$ is 
{\it agreeable}, namely is of the form 
$\af=\bt^k \bt'$ for some $k\geq 0$ and some paths $\bt, \bt'\in \CL(E^{\leq l})$,  
where $\bt'$ is an initial path of $\bt$. 
 In case  of trivial labeling, $(E,\CL_{id}, \bEz)$ is disagreeable 
 exactly when the graph $E$ satisfies condition (L) \cite[Lemma 5.3]{BP1}.  
 Condition (L) meaning that every loop has an exit was introduced 
 as an essential hypothesis for  the Cuntz-Krieger uniqueness theorem 
 for graph $C^*$-algebras in \cite{KPR}.  
 More generally, in \cite[Theorem 5.5]{BP1}  
 it is known that if $(E,\CL, \bEz)$ is disagreeable, 
 Cuntz-Krieger uniqueness Theorem holds, that is, 
 if $\{S_a, P_A\}$ is a representation of a labeled space $(E,\CL,\bEz)$  
 such that $S_a\neq 0$ and $P_A\neq 0$ for all $a\in \CA$ and 
 $A\in \bEz$, the $C^*$-algebra $C^*(S_a,P_A)$ generated by $\{S_a, P_A\}$ 
 is isomorphic to $C^*(E,\CL,\bEz)$.  
 
As pointed out in \cite{BP2}, a disagreeable labeled space $(E,\CL,\bEz)$  
contains lots of aperiodic paths and in fact, as can be seen in the following proposition,   
$(E,\CL,\bEz)$ is disagreeable whenever it has no repeatable paths. 
 
\vskip 1pc 

\begin{prop}  Let $E$  be a  directed graph with no sinks or sources. 
If the labeled space $(E,\CL,\bEz)$ has no repeatable paths, it is always disagreeable.
\end{prop}  

\begin{proof} Assuming that $(E,\CL,\bEz)$ is not disagreeable, 
one can pick a generalized vertex  $[v]_l$ that is not disagreeable. 
Then there is an $N>0$ such that every $\af$ in $\CL([v]_l E^{\geq N})$ is 
agreeable and of the form $\af=\bt^k\bt'$ for some 
$\bt\in \CL([v]_l E^{\leq l})$ and its initial path $\bt'$.  
 On the other hand, there are only finitely many labeled paths in $\CL([v]_l E^{\leq l})$  
while $\CL([v]_l E^{\geq N})$ has infinitely many labeled paths. 
This shows that there should exist a  path $\bt$ in $\CL([v]_l E^{\leq l})$ 
such that  its repetitions $\bt^n$ appear  in  $\CL([v]_l E^{\geq N})$ 
for all sufficiently large $n$.   
\end{proof}

\vskip 1pc 

One might expect that a labeled space to be disagreeable if it has no loops, 
but this is not necessarily true; 
see the following example.  

\vskip 1pc

\begin{ex}\label{ex-final} 
Consider the following labeled graph 
$(E,\CL)$:
\vskip 1pc
\hskip 3pc
 \xy /r0.34pc/:(-35,0)*+{\cdots};(35,0)*+{\cdots ,};
(-30,0)*+{\bullet}="V-3";
(-20,0)*+{\bullet}="V-2";
(-10,0)*+{\bullet}="V-1"; (0,0)*+{\bullet}="V0";
(10,0)*+{\bullet}="V1"; (20,0)*+{\bullet}="V2";
(30,0)*+{\bullet}="V3";
 "V-3";"V-2"**\crv{(-30,0)&(-20,0)};
 ?>*\dir{>}\POS?(.5)*+!D{};
 "V-2";"V-1"**\crv{(-20,0)&(-10,0)};
 ?>*\dir{>}\POS?(.5)*+!D{};
 "V-1";"V0"**\crv{(-10,0)&(0,0)};
 ?>*\dir{>}\POS?(.5)*+!D{};
 "V0";"V1"**\crv{(0,0)&(10,0)};
 ?>*\dir{>}\POS?(.5)*+!D{};
 "V1";"V2"**\crv{(10,0)&(20,0)};
 ?>*\dir{>}\POS?(.5)*+!D{};
 "V2";"V3"**\crv{(20,0)&(30,0)};
 ?>*\dir{>}\POS?(.5)*+!D{};
 (-25,1.5)*+{-3};
 (-15,1.5)*+{-2};(-5,1.5)*+{-1};(5,1.5)*+{0};
 (15,1.5)*+{0};(25,1.5)*+{0}; 
 (0.1,-2.5)*+{v_0};(10.1,-2.5)*+{v_1};
 (-9.9,-2.5)*+{v_{-1}};
 (-19.9,-2.5)*+{v_{-2}};
 (-29.9,-2.5)*+{v_{-3}};
 (20.1,-2.5)*+{v_{2}};
 (30.1,-2.5)*+{v_{3}}; 
\endxy 

\vskip 1pc
\noindent 
Then $\bEz$ is the collection of all finite sets $F$ of  $E^0$  
and sets of the form $F\cup \{v_n,v_{n+1}, \dots\}$, $n\geq 1$. 
For the generalized vertex $\{v_0\}=[v_0]_1$, every path 
$\af\in \CL(v_0E^{\geq N})$ is agreeable since it must be equal to 
$a^m$ for some $m\geq N$, so the  labeled space $(E,\CL,\bEz)$ 
is not disagreeable, whereas  it is obvious that$(E,\CL,\bEz)$ has no loops.
\end{ex}

\end{document}